\documentclass[12pt]{article}
\usepackage{a4wide}

\usepackage[a4paper, margin=2.3cm]{geometry}
\usepackage{amsmath,amssymb,amsthm, comment}
\usepackage{color}
\usepackage{xcolor}
\usepackage{parskip}
\usepackage{hyperref}
\usepackage{parskip}
\usepackage{setspace}
\usepackage{graphicx}
\usepackage{mathtools}
\usepackage[thinc]{esdiff}
 
\newtheorem{theorem}{Theorem}[section]

\newtheorem{corollary}[theorem]{Corollary}
\newtheorem{lemma}[theorem]{Lemma}
\newtheorem{question}[theorem]{Question}

\newtheorem{proposition}[theorem]{Proposition}

\newtheorem*{definition*}{Definition}

\def\C{\mathcal{C}}

\def\B{\mathcal{B}}
\def\E{E}

\def\E{\mathbb{E}}

\def\C{\mathcal{C}}

\def\C{\mathcal{C}}

\def\B{\mathcal{B}}

\begin{document}

\title{Geometric structures in pseudo-random graphs}
\author{Thang Pham\thanks{University of Science, Vietnam National University, Hanoi. Email: thangpham.math@vnu.edu.vn}\and Steven Senger\thanks{Department of Mathematics, Missouri State University. Email: StevenSenger@MissouriState.edu} \and Michael Tait\thanks{Department of Mathematics \& Statistics, Villanova University. Email: michael.tait@villanova.edu}\and Vu Thi Huong Thu\thanks{University of Science, Vietnam National University, Hanoi. Email: VuThiHuongThu\_T64@hus.edu.vn}}
\maketitle
\begin{abstract}
In this paper, we provide a general framework for counting geometric structures in pseudo-random graphs. As applications, our theorems recover and improve several results on the finite field analog of questions originally raised in the continuous setting. The results present interactions between discrete geometry, geometric measure theory, and graph theory.
\end{abstract}
\tableofcontents
\section{Introduction}
Let $\mathbb{F}_q$ be a finite field of order $q$ where $q$ is a prime power. The investigation of finite field analogs of problems originally raised in geometric measure theory has a long tradition, for instance, the Erd\H os-Falconer distance problem \cite{EF1, EF2, EF3}, sum-product estimates \cite{SP1, SP2}, the Kakeya problem \cite{KP1, KP2}, frame theory \cite{FT1, FT2}, and restriction problems \cite{RT5, RT3,  RT2, RT1, RT4}. Studying these problems over finite fields is not only interesting by itself, but it also offers new ideas to attack the original questions. Some of these problems can be proved by using results from the in graph theory, for instance, in \cite{EF3}, Iosevich and Rudnev proved the following theorem on the distribution of distances in a given set. 
\begin{theorem}[Iosevich-Rudnev, \cite{EF3}]\label{distance-1}
Let $E$ be a set in $\mathbb{F}_q^d$. Assume that $|E|\gg q^{\frac{d+1}{2}}$, then $$\Delta(E):=\{||x-y||=(x_1-y_1)^2+\cdots+(x_d-y_d)^2\colon x, y\in E\}=\mathbb{F}_q.$$
\end{theorem}
It is well-known that this theorem can be reproved by using the famous expander mixing lemma. More precisely, the expander mixing lemma states that for an $(n, d, \lambda)$-graph $\mathcal{G}$, i.e. a regular graph with $n$ vertices, of degree $d$, and all other eigenvalues bounded in absolute value by $\lambda$, the number of edges in a given vertex set $U$, denoted by $e(U)$, is bounded from both above and below by the inequality
\[\left\vert e(U)-\frac{d|U|^2}{2n} \right\vert\le \lambda |U|.\]
To derive Theorem \ref{distance-1} from this estimate, for $\alpha\in \mathbb{F}_q^*$, one just needs to define the distance graph $\mathcal{DG}_\alpha$ with the vertex set $\mathbb{F}_q^d$ and there is an edge between two vertices $x$ and $y$ if and only if $||x-y||=\alpha$. It is not hard to check that $\mathcal{DG}_\alpha$ is a regular Cayley graph with $q^d$ vertices, of degree $\sim q^{d-1}$, and the second eigenvalue is bounded by $q^{\frac{d-1}{2}}$ by using Kloosterman sums \cite{KLS2, KLS1}. So, for any $\alpha\ne 0$, the expander mixing lemma implies directly that any vertex set $U$ of size at least $2q^{\frac{d+1}{2}}$ spans at least one edge. 

We observe that the argument above only made use of the pseudo-randomness properties of the graph, and once the eigenvalues were calculated ignored anything about $\mathbb{F}_q$ or the distance function. Because of this observation, this machinery provides a unified proof for a series of similar questions, for example, one can replace the distance function by bilinear forms \cite{Hart}, Minkowski distance function \cite{HIS}, or other functions \cite{Vinh}.

From this observation, it is very natural to ask what kind of finite field models can be extended to the graph setting? That is, what ``geometric structures" can we guarantee in a general graph with some pseudo-randomness condition? The main purpose of this paper is to provide three such configurations, and the three topics we present here can be viewed as generalizations of the Erd\H{o}s-Falconer distance conjectures, which have been studied intensively in the literature. Our theorems imply several results found previously as special cases. Moreover, as they rely on the pseudo-randomness of an underlying graph, they can be applied in a straightforward manner to other settings, such as modules over finite rings.

Throughout the paper we say that $G$ is an $(n, d, \lambda)$-colored graph with color set $D$ if it is a graph edge-colored with $|D|$ colors such that the subgraph of any fixed color is an $(n, d, \lambda)$-graph.

\subsection{Cartesian product structures}
We first start with the following question about finding rectangles in $\mathbb{F}_q^d$. 
\begin{question}
Let $E$ be a set in $\mathbb{F}_q^d$ and $\lambda, \beta\in \mathbb{F}_q^*$. How large does $E$ need to be to guarantee that there are four points $w, x, y, z\in E$ such that they form a rectangle of side lengths $\alpha$ and $\beta$, i.e.
\begin{equation}\label{eqmot}(w-x)\cdot (x-y)=0, ~(x-y)\cdot (y-z)=0, ~(y-z)\cdot(z-w)=0, ~(z-w)(w-x)=0,\end{equation}
and \begin{equation}\label{eqhai}||w-x||=||y-z||=\alpha, ~||x-y||=||z-w||=\beta.\end{equation}
\end{question}
Lyall and Magyar \cite{LM1} proved that for any $\delta\in (0, 1)$, there exists an integer $q_0=q_0(\delta)$ with the following property: if $q\ge q_0$ and $E\subset \mathbb{F}_q^{2d}$ with $|E|\ge \delta q^{2d}$, then $E$ contains four points $a, b, c$, and $d$ satisfying (\ref{eqmot}) and (\ref{eqhai}). This is the finite field model of a result in the same paper which states that for any given  rectangle $\mathcal{R}$ in $\mathbb{R}^{2d}$, if $S\subset \mathbb{R}^{2d}$ has positive Banach density, then there exists a threshold $\lambda_0=\lambda_0(S, \mathcal{R})$ such that $S$ contains an isometric copy of $\lambda \mathcal{R}$ for any $\lambda \ge \lambda_0$. Notice that the result in \cite{LM1} was actually proved in a more general form, for $d$-dimensional rectangles, though we state it here for $2$-dimensional rectangles. 

In the first theorem of this paper, we extend this result to a general graph setting. 

For two graphs $G$ and $H$, the {\em cartesian product of $G$ and $H$}, denoted by $G\square H$, is the graph where $V(G\square H) = V(G)\times V(H)$ and $(u_1, v_1) \sim (u_2, v_2)$ if and only if either $u_1=u_2$ and $\{v_1, v_2\} \in E(H)$ or $v_1=v_2$ and $\{u_1, u_2\} \in V(G)$. We use $S(x)$ to denote the indicator function of the set $S.$

\begin{theorem}\label{thm:CS}
Let $G_i$ be $(n_i ,d_i, \lambda_i)$-graphs with $1\le i\le 2$. Set $G=G_1\square G_2$. 
For any $0<\delta'<\delta<1$, there exists 
$\epsilon>0$ such that for any $S\subset V(G_1\square G_2)$ with $|S|\ge \delta |V(G_1\square G_2)|$, if $\max\left\{ \frac{\lambda_1}{d_1}, \frac{\lambda_2}{d_2}\right\} < \epsilon$, then 
\[N=\sum_{(u_1, u_2)\in E(G_1), (v_1,v_2)\in E(G_2)}S(u_1, v_1)S(u_1, v_2)S(u_2, v_1)S(u_2, v_2)>\delta'^4 n_1n_2d_1d_2.\]
\end{theorem}

Theorem \ref{thm:CS} recovers the theorem on rectangles in $\mathbb{F}_q^d$ by Lyall and Magyar (Proposition 2.1 in \cite{LM1}) as follows. Let $G_1$ and $G_2$ be the graphs each with vertex set $\mathbb{F}_q^d$ where $a\sim b$ in $G_1$ if $||a-b|| = \alpha$ and $x\sim y$ in $G_2$ if $||a-b|| = \beta$. Then $G_1$ and $G_2$ are graphs with $q^d$ vertices, degree asymptotically $q^{d-1}$ and $\lambda \leq 2 q^{(d-1)/2}$ (see \cite{BST, kwok}, summarized as Theorem 10.1 in \cite{Vinh}). Note that if $||u_1- u_2|| = \alpha$ and $||v_1-v_2|| = \beta$, then letting $w = (u_1, v_1)$, $x = (u_1, v_2)$, $y=(u_2, v_1)$ and $z = (u_2, v_2)$, we have that $w,x,y,$ and $z$ form a rectangle in $\mathbb{F}_q^{2d}$ with side lengths $\alpha$ and $\beta$. Applying Theorem \ref{thm:CS} to these specific graphs shows that for $q$ large enough, any subset of $\mathbb{F}_q^{2d}$ of size at least $\delta q^{2d}$ contains $\Omega(q^{4d-2})$ rectangles, giving a quantitative strengthening of Lyall and Magyar's result. Another application of Theorem \ref{thm:CS} is on the number of rectangles in $\mathbb{F}_q^2$ with side lengths in a given multiplicative subgroup of $\mathbb{F}_q$, precisely, given a multiplicative subgroup $A$ of $\mathbb{F}_q$, we define $G_1=G_2$ being the graph with the vertex set $\mathbb{F}_q$ and there is an edge between $x$ and $y$ if $x-y\in A$. This is clear that this is a Cayley graph with $q$ vertices, of degree $|A|$, and it is also well-known that $\lambda\le q^{1/2}$ (see \cite[(1)]{koh2021configurations} for computations). Applying Theorem \ref{thm:CS}, we recover Theorem 1.1 from \cite{koh2021configurations}. 

\subsection{Distribution of cycles}
Our motivation of this section comes from the following question. 
\begin{question}
Let $E$ be a set in $\mathbb{F}_q^d$ and $m\ge 4$ be an integer. How large does $E$ need to be to guarantee that the number of cycles of the form $(x_1, \ldots, x_m)$ with $||x_i-x_{i+1}||=1$ for all $1\le i\le m-1$, and $||x_m-x_1||=1$, is close to the expected number $|E|^mq^{-m}$?
\end{question}
Iosevich, Jardine, and McDonald \cite{IJM} proved that the number of cycles of length $m$, denoted by $C_m(E)$, satisfies 
\begin{equation}\label{cycle}C_m(E)=(1+o(1))\frac{|E|^m}{q^m},\end{equation}
whenever
\begin{align*}
|E|\geq \left\{\begin{array}{ll}
q^{\frac{1}{2}\left(d+2-\frac{k-2}{k-1}+\delta\right)} & : m=2k, \ \text{even} \\
q^{\frac{1}{2}\left(d+2-\frac{2k-3}{2k-1}+\delta\right)} & :m=2k+1 \ \text{odd}
\end{array}
\right.
\end{align*}
where 
\begin{align*}
0<\delta<\frac{1}{2\left\lfloor\frac{m}{2}\right\rfloor^2}.
\end{align*}
In the continuous setting, this is a difficult problem, and there are only a few partial results. For instance, as a consequence of a theorem due to Eswarathasan, Iosevich, and Taylor \cite{EIT}, we know that if the Hausdorff dimension of $E$, denoted by $s$, is at least $\frac{d+1}{2}$, then we know that the upper Minkowski dimension of the set of cycles in $E$ is at most $2s-m$. If we consider the case of paths, then Bennett, Iosevich, and Taylor \cite{BIT} showed that there exists an open interval $I$ such that for any sequence $\{t_i\}_{i=1}^m$ of elements in $I$, we always can find paths of length $m+1$ with gaps $\{t_i\}_{i=1}^m$ between subsequent elements in $E$ as long as the Hausdorff dimension of $E$ is greater than $\frac{d+1}{2}$. We refer the reader to \cite{liu2020hausdorff} for the recent study on this problem. It is worth noting that in the discrete setting, results on distribution of paths also play crucial role in proving (\ref{cycle}).

In the graph setting, we have the following extension. We note here that we are counting any sequence of $m$ vertices $(v_1,\cdots, v_m)$ with $v_i\sim v_{i+1}$ and $v_1\sim v_m$ as a cycle of length $m$. That is, we are counting labeled cycles and we include degenerate cycles in the count. One could combine Theorem \ref{cycle-main1} for various values of $m$ and lemmas used to prove it to obtain results about non-degenerate cycles as well, but we do not do this explicitly here.

\begin{theorem}\label{cycle-main1}
Let $\mathcal{G}$ be an $(n, d, \lambda)$-graph and $U$ be a vertex set with ${\lambda \cdot \frac{n}{d}=o(|U|)}$. Let $C_m(U)$ denote the number of (labeled, possibly degenerate) cycles of length $m$ with vertices in $U$. Then we have 
\[
\left|C_m(U) -\frac{|U|^md^m}{n^m}\right| = O \left( \frac{\lambda |U|^{m-1}d^{m-1}}{n^{m-1}} + \frac{\lambda^{m-2} |U|^2 d}{n}\right),
\]
\end{theorem}
The error term cannot be improved for $m=4$. For instance, we define a graph with the vertex set $\mathbb{F}_q^2$ where two vertices $(a, b)$ and $(c, d)$ are adjacent if and only if $ac+bd=1$. Using the geometric facts in $\mathbb{F}_q^2$ that any two lines intersect in at most one point and there is only one line passing through two given points, we can see that this graph contains no $C_4$, even though it is a $(q^2, q, \sqrt{q})$ graph (if one includes loops). We also remark that as a corollary of a result due to Alon \cite[Theorem 4.10]{MSS} we know that the number of cycles in $U$ is close to the expected number as long as $|U|\gg \lambda (n/d)^2$. This result is of course weaker than Theorem \ref{cycle-main1}.

We now discuss how Theorem \ref{cycle-main1} implies and improves previous results. In \cite{IJM}, counting results for cycles are proved in both the distance graph and the dot-product graph over $\mathbb{F}_q^d$. Formally, let $G_t^{dist}$ and $G_t^{prod}$ be the graphs on vertex set $\mathbb{F}_q^d$ where $u\sim v$ in $G_t^{dist}$ if $||u-v|| = t$ and $u\sim v$ in $G_t^{prod}$ if $u\cdot v = t$. As each of these graphs are approximately $q^{d-1}$ regular and with second eigenvalue bounded above by $2q^{(d-1)/2}$, Theorem \ref{cycle-main1} can be applied. In \cite{IJM}, the same quantitative results are proved for both graphs but with different methods, and the authors write the following:

``We note that in this paper, we obtain the same results for the distance graph and the dot-
product graph. While the techniques are, at least superficially, somewhat different due to the
lack of translation invariance in the dot-product setting, it is reasonable to ask whether a general
formalism is possible."

Theorem \ref{cycle-main1} answers this question in a strong way, as it may be applied in a much more general setting than just distance or dot-product graphs. Furthermore, Theorem \ref{cycle-main1} implies the estimate \eqref{cycle} with an improved threshold on the size of the subset, namely we may remove the $\delta$ in the exponent that appears in the result from \cite{IJM}.  The proof of Theorem \ref{cycle-main1} requires estimates on the number of paths in our graph, for example Proposition \ref{paths}. This is again done in a general way for $(n, d, \lambda)$-graphs. We also note that colorful versions of Theorem \ref{cycle-main1} and the lemmas required to prove it can be proved with only minor modifications to the proof. That is, given an $(n, d, \lambda)$-colored graph and a fixed coloring of a path or cycle, one can obtain the same estimates on the number of such colorful subgraphs that appear. For ease of exposition we only prove an uncolored version of Theorem \ref{cycle-main1}, but Theorems \ref{tree-thm1} and \ref{tree-thm12} (see below) are stated and proved in a colorful way as proof of concept. It is possible through this general set up to recover Theorem 1.1 of \cite{bene} and Theorem 6 of \cite{steven}.

Finally, we prove Theorem \ref{cycle-main1} in two different ways. The second approach is quite specific to counting cycles, but more straightforward (it is also slightly weaker: we obtain the same quantitative results for $m\geq 5$ but for $m=4$ only prove the result up to a multiplicative constant factor). The first approach passes the problem to counting structures in the tensor product of two $(n, d, \lambda)$-graphs. We note that the tensor product of two $(n, d, \lambda)$-graphs is itself a $(n^2, d^2, d\lambda)$ graph, and so one may try to use pseudo-randomness of this graph to count subgraphs. However, this is not good enough for our purpose, and we must prove a version of the expander mixing lemma that applies specifically to tensor products of graphs. This result (Proposition \ref{keylemma}) is significantly stronger than directly applying the classical expander mixing lemma to the tensor product graph, and we believe it is of independent interest, as the second approach along with Proposition \ref{keylemma} could be used to count other structures in tensor products of pseudo-random graphs.

\subsection{Distribution of disjoint trees}
The last question we consider in this paper is the following.
\begin{question}
Let $E$ be a set in $\mathbb{F}_q^d$, and $T$ be a tree of $m$ vertices. How large does $E$ need to be to guarantee that the number of vertex disjoint copies of $T$ in $E$ is close to $|E|/m$? 
\end{question}

That is, we are asking for a threshold such that any set of large enough size has an almost spanning $T$-factor. We now to introduce the notion of the stringiness of a graph, $T$, denoted $\sigma(T),$ which is defined as $(d_1+1)\prod_{i=2}^n d_i$ where $d_1\geq d_2 \cdots \geq d_n$ is the degree sequence of $T$ in nonincreasing order. Using this, Soukup \cite{david} proved that for any tree $T$ of $m$ vertices with stringiness $\sigma(T)$, and for any $E\subset \mathbb{F}_q^d$, if $|E|\gg \sigma(T)q^{\frac{d+1}{2}}$, then the number of disjoint copies of $T$ in $E$ is at least 
\[\frac{|U|}{\sigma(T)}-q^{\frac{d+1}{2}}.\]

In this section, we provide improvements of this result. 

\begin{theorem}\label{tree-thm1}
Let $G$ be an $(n, d, \lambda)$-colored graph with the color set $D$. Let $T$ be a tree with edges colored by $D$. For any $U\subset V(G)$ with $|U|=r\cdot \frac{\lambda n}{d}$, the number of disjoint copies of $H$ in $U$ is at least 
\[\frac{|U|}{\sigma(T)}-\frac{\lambda n}{d},\]
where $\sigma(T)$ is the stringiness of $T$.
\end{theorem}

Theorem \ref{tree-thm1} directly generalizes Soukup's result in \cite{david} to pseudo-random graphs. However, the stringiness of a tree may be exponential in the number of vertices. Using a different method, we prove a theorem which for most trees does much better.

\begin{theorem}\label{tree-thm12}
Let $G$ be an $(n, d, \lambda)$-colored graph with the color set $D$. Let $T$ be a tree of $m$ vertices with edges colored by $D$. For any $U\subset V(G)$ with $|U|\ge m(m-1)\cdot \frac{\lambda n}{d}$, the number of disjoint copies of $T$ in $U$ is at least 
\[\frac{|U|}{m}-\frac{\lambda n}{d}.\]
\end{theorem}

\section{Proof of Theorem \ref{thm:CS}}
Set $V_i=V(G_i)$ and $E_i=E(G_i)$ for $1\le i\le 2$. If $i$ satisfies $\frac{\lambda_i}{d_i} = \max \left\{ \frac{\lambda_1}{d_1}, \frac{\lambda_2}{d_2}\right\}$ then throughout the proof we will use $\frac{\lambda}{d}$ to denote $\frac{\lambda_i}{d_i}$.

\subsection{Square-norm}
For functions $f_1, f_2, f_3, f_4\colon V_1\times V_2\to [-1,1]$, we define
\begin{align*}N(f_1, f_2, f_3, f_4):=&\E_{\substack{a, b, c, d\\(a, b)\in E_1, (c, d)\in E_2}}f_1(a, c)f_2(a, d)f_3(b, c)f_4(b, d)\\:= \frac{1}{|V_1|d_1|V_2|d_2}&\sum_{\substack{a, b, c, d\\(a, b)\in E_1, (c, d)\in E_2}}f_1(a, c)f_2(a, d)f_3(b, c)f_4(b, d), \end{align*}
and 
\begin{align*}
M(f_1, f_2, f_3, f_4):=&\E_{a, b, c, d}f_1(a, c)f_2(a, d)f_3(b, c)f_4(b, d)\\:= \frac{1}{|V_1|^2|V_2|^2} &\sum_{a, b, c, d}f_1(a, c)f_2(a, d)f_3(b, c)f_4(b, d).
\end{align*}

Let $S$ be any subset of $V_1\times V_2$. Recall that when context is clear, we use $S(\cdot)$ to denote the characteristic function $\chi_S$ on the set $S$. We now prove two simple but useful facts about $M$ using Cauchy-Schwarz.

\begin{proposition}\label{MCS}
$$M(S,S,S,S) \geq \left(\frac{|S|}{|V_1||V_2|}\right)^4.$$
\end{proposition}
\begin{proof}
We write the definition of $|S|$ as a sum and apply Cauchy-Schwarz twice to get
\begin{align*}
|S| &= \sum_{a\in V_1}\sum_{b\in V_2}S(a,b) \leq \left(\sum_{a\in V_1}1^2\right)^\frac{1}{2}\left(\sum_{a\in V_1}\left(\sum_{b\in V_2}S(a,b)\right)^2\right)^\frac{1}{2}\\
&= |V_1|^\frac{1}{2}\left(\sum_{b\in V_2}\sum_{c\in V_2}\sum_{a\in V_1}S(a,b)S(a,c)\right)^\frac{1}{2}\\
&\leq |V_1|^\frac{1}{2}\left(\left(\sum_{b\in V_2}\sum_{c\in V_2}1^2\right)^\frac{1}{2}\left(\sum_{b\in V_2}\sum_{c\in V_2}\left(\sum_{a\in V_1}S(a,b)S(a,c)\right)^2\right)^\frac{1}{2}\right)^\frac{1}{2}\\
&= |V_1|^\frac{1}{2}\left(|V_2|\left(\sum_{b\in V_2}\sum_{c\in V_2}\sum_{a\in V_1}\sum_{d\in V_1}S(a,b)S(a,c)S(d,b)S(d,c)\right)^\frac{1}{2}\right)^\frac{1}{2},
\end{align*}
which, upon rearranging and renaming variables becomes
$$|V_1|^\frac{1}{2}|V_2|^\frac{1}{2}\left(|V_1|^2|V_2|^2M(S,S,S,S) \right)^\frac{1}{4}.$$
Comparing this to $|S|$ yields the desired result.
\end{proof}

For any function $f\colon V_1\times V_2\to [-1, 1],$  we define 
\[||f||_{\square(V_1\times V_2)}:=M(f, f, f, f)^{1/4}.\]
\begin{lemma}\label{basic1}
For functions $f_1, f_2, f_3, f_4\colon V_1\times V_2\to [-1, 1]$, we have 
\[ M(f_1, f_2, f_3, f_4)\le \min_{i}||f_i||_{\square(V_1\times V_2)}.\]
\end{lemma}
\begin{proof}
We apply Cauchy-Schwarz to the definition of $M$ to get
\begin{align*}
M(f_1,f_2,f_3,f_4) &= \frac{1}{|V_1|^2|V_2|^2} \sum_{a, b\in V_1, c, d\in V_2}f_1(a,c)f_2(a,d)f_3(b,c)f_4(b,d)\\
&= \frac{1}{|V_1|^2|V_2|^2} \sum_{a, b\in V_1}\left(\sum_{c\in V_2}f_1(a,c)f_3(b,c)\right) \left(\sum_{d\in V_2}f_2(a,d)f_4(b,d)\right)\\
&\leq \frac{1}{|V_1|^2|V_2|^2} \left(\sum_{a, b\in V_1}\left(\sum_{c\in V_2}f_1(a,c)f_3(b,c)\right)^2\right)^\frac{1}{2}\\
&\qquad\qquad \cdot\left(\sum_{a, b\in V_1}\left(\sum_{d\in V_2}f_2(a,d)f_4(b,d)\right)^2\right)^\frac{1}{2}\\
&=\left(M(f_1,f_1,f_3,f_3) \right)^\frac{1}{2} \cdot \left(M(f_2,f_2,f_4,f_4) \right)^\frac{1}{2}.
\end{align*}
A similar calculation using Cauchy-Schwarz and reversing the roles of $V_1$ and $V_2$ gives that 
\[
M(f_1, f_2, f_3, f_4) \leq (M(f_1, f_2, f_1, f_2))^{1/2}\cdot (M(f_3, f_4, f_3, f_4))^{1/2}.
\]
We finish by combining these inequalities and using the fact that $M(f_i,f_i,f_i,f_i) \leq 1$ for $i=1,2,3,4.$
\end{proof}

\subsection{A weak hypergraph regularity lemma}
Let $\B$ be a $\sigma$-algebra on $V_1$ and $\C$ be a $\sigma$-algebra on $V_2$. We recall here that a $\sigma$-algebra on $V_i$ is a collection of sets in $V_i$ that contains $V_i$, $\emptyset$, and is closed under finite intersections, unions, and complements.  
 
The complexity of a $\sigma$-algebra $\B$ is the smallest number of sets (atoms) needed to generate $\B$, and we denote by $\mathtt{complexity}(\B)$. Notice that $|\B|\le 2^{\mathtt{complexity}(\B)}$. We denote the smallest $\sigma$-algebra on $V_1\times V_2$ that contains both $\B\times V_2$ and $V_1\times \C$ by $\B\vee \C$. 

For a function $f\colon V_1\times V_2 \to \mathbb{R}$, we define the conditional expectation $\E(f|\B\vee \C)\colon V\to \mathbb{R}$ by the formula 
\[\E(f|\B\vee \C)(x):=\frac{1}{|(\B\vee \C)(x)|}\sum_{y\in (\B\vee \C)(x)}f(y),\]
where $(\B\vee \C)(x)$ denotes the smallest element of $\B\vee \C$ that contains $x$. We note that an atom of $\B\vee\C$ has the form $U\times V$ where $U$ and $V$ are atoms of $\B$ and $\C$, respectively.

The following lemma is a special case of Lemma 2.2 in \cite{LM1}. We refer the reader to \cite{LM1} for a detailed proof. 
\begin{lemma}\label{co2.2}
For any $\epsilon>0$, there exist $\sigma$-algebras $\B$ on $V_1$ and $\C$ on $V_2$ such that each algebra is spanned by at most $O(\epsilon^{-8})$ sets, and 
 \[||S-\E(S| \B\vee\C)||_{\square(V_1\times V_2)}\le \epsilon.\]
\end{lemma}
We recall that
 \[\E(S|\B\vee\C)(x)=\frac{|S\cap (B\times C)|}{|B||C|},\]
where $B\times C$ is the atom of $\B\vee \C$ containing $x$.

\subsection{A generalized von-Neumann type estimate}

\begin{lemma}\label{co2.31}
For functions $f_1, f_2, f_3, f_4\colon V_1\times V_2\to [-1, 1]$, we have 
\[|N(f_1, f_2, f_3, f_4)|\le \min_{j}||f_j||_{\square (V_1\times V_2)}+O\left(\frac{\lambda^{1/4}}{d^{1/4}}\right).\]
\end{lemma}
To prove this lemma, we recall the following expander mixing lemma.
\begin{lemma}\label{th:expanderMixing}
Let $G = (V,E)$ be an $(n, d, \lambda)$-graph, and $A$ be its adjacency matrix. For real $f, g\in L^2(V)$, we have 
\[\left|\langle
f,Ag\rangle-d|V|\mathbb{E}(f)\mathbb{E}(g)\right|\leq \lambda\|f\|_2\|g\|_2,\]
where 
\[\mathbb{E}(f):=\frac{1}{|V|}\sum_{v\in V}f(v), ~||f||_2^2=\sum_{v\in V}|f(v)|^2.\]
\end{lemma}
\begin{proof}[Proof of Lemma \ref{co2.31}]
Set $\sigma_i(x, y)=|V_i|/d_i$ if $(x, y)\in E_i$ and $0$ otherwise and let $\E_{x,y\in V_i} := \frac{1}{|V_i|^2}\sum_{x,y\in V_i}$. Then for $f, g\colon V_i\to [-1, 1]$, by using the expander mixing lemma, one has 
\[
\sum_{x\sim y} f(x)g(y) =  \langle f, Ag \rangle \leq \frac{d_i}{|V_i|} \sum_{x,y} f(x)g(y) + \lambda_i \|f\|_2\|g\|_2.
\]
Dividing both sides by $d_i|V_i|$ and using $\|f\|_2\|g\|_2 \leq |V_i|$ gives

\[\E_{x, y\in V_i}f(x)g(y)\sigma_i(x, y)=\left(\frac{1}{|V_i|^2}\sum_{x, y}f(x)g(y)\right)+\frac{\lambda_i}{d_i} = \E_{x,y} f(x)f(y) + \frac{\lambda_i}{d_i}.\]
Thus,
\begin{align*}
|\E_{x, y}f(x)g(y)\sigma_i(x,y)|^2&\leq (\E_{x, y, z, t}f(x)g(z)f(y)g(t)) + 2\frac{\lambda_i}{d_i} \E_{x,y}f(x)g(y) + \frac{\lambda_i^2}{d_i^2}\\
&\leq \E_{x,y} f(x)f(y) + 3\frac{\lambda_i}{d_i},
\end{align*}
where we have used the fact that $\E_{z,t}g(z)g(t), \E_{x,y} f(x)g(y) \leq 1$. In other words, for functions $f,g: V_i \to [-1,1],$ we have
\begin{equation}\label{eq3}
|\E_{x, y}f(x)g(y)\sigma_i(x,y)|^2 \le  \E_{x, y}f(x)f(y)+ 3\frac{\lambda_i}{d_i}.
\end{equation}
The same holds when we switch between $f$ and $g$:
\begin{equation}\label{eq4}
|\E_{x, y}f(x)g(y)\sigma_i(x,y)|^2 \le  \E_{x, y}g(x)g(y)+ 3\frac{\lambda_i}{d_i}.
\end{equation}
In the next step, we want to show that \[N(f_1, f_2, f_3, f_4)\le ||f_1||_{\square(V_1\times V_2) }+{O\left(\frac{\lambda}{d}\right)}.\]  Using the definitions, we have
\begin{align*}
N(f_1, f_2, f_3, f_4)&=\E_{a, b, c, d}f_1(a, c)f_2(a, d)f_3(b, c)f_4(b, d)\sigma_1(a,b)\sigma_2(c,d).
\end{align*}
For a fixed pair $(c, d)$, set $f_{c,d}(a)=f_1(a, c)f_2(a, d)$ and $g_{c,d}(b)=f_3(b, c)f_4(b, d).$ Then we have 
\begin{align*}
|N(f_1, f_2, f_3, f_4)|^2 & = \left( \frac{1}{|V_1|^2|V_2|^2} \sum_{a,b,c,d} f_{c,d}(a) g_{c,d}(b) \sigma_1(a,b)\sigma_2(c,d)\right)^2\\
&\leq \left(\frac{1}{|V_1|^2|V_2|^2} \sum_{c,d} \sqrt{\sigma_2(c,d)} \left| \sqrt{\sigma_2(c,d)} \sum_{a,b} f_{c,d}(a)g_{c,d}(b) \sigma_1(a,b)\right| \right)^2\\
& \leq \frac{1}{|V_1|^4 |V_2|^4} \left( \sum_{c,d} \sigma_2(c,d) \right) \left( \sum_{c,d} \sigma_2(c,d) \left| \sum_{a,b}  f_{c,d}(a)g_{c,d}(b) \sigma_1(a,b)\right|^2\right) \\
&= \left( \E_{c,d} \sigma_2(c,d)\right) \left( \E_{c,d}  \sigma_2(c,d) | \E_{a,b}  f_{c,d}(a)g_{c,d}(b) \sigma_1(a,b)|^2\right) &
\\ &=\E_{c,d} \sigma(c,d) | \E_{a,b} f_{c,d}(a)g_{c,d}(b) \sigma(a,b)|^2,
\end{align*}
where the first inequality uses the triangle inequality and rearranging, the second inequality is Cauchy-Schwarz, the next line is rearranging, and the last equality uses the fact that $\E_{c, d}\sigma_2(c,d)=\frac{1}{|V_2|^2}\cdot \frac{|V_2|}{d_2}\cdot |V_2|\cdot d_2=1$.
Therefore, the inequality \eqref{eq3} implies
\begin{align} \label{eq5K}|N(f_1, f_2, f_3, f_4)|^2& \leq \E_{c,d} \sigma_2(c,d) \left(\E_{a,b} f_{c,d}(a)f_{c,d}(b) + 3\frac{\lambda_1}{d_1} \right)\\ &
= \E_{a,b,c,d} \sigma_2(c,d) f_{c,d}(a)f_{c,d}(b) +  3\left(\E_{c,d} \sigma_2(c,d)\right)\frac{\lambda_1}{d_1} \nonumber \\ 
&=\left(\E_{a, b, c, d}f_1(a, c)f_1(b, c)f_2(a, d)f_2(b, d)\sigma_2(c,d)\right)+3\frac{\lambda_1}{d_1} \nonumber.\end{align}

By another similar argument with $\hat{f}_{a,b}(c)=f_1(a, c)f_1(b, c)$ and $\hat{g}_{a,b}(d)=f_2(a, d)f_2(b, d)$ for each fixed pair $(a, b)$, we have
\begin{align*}
    |N(f_1, f_2, f_3, f_4)|^4 &\leq \left( \E_{a,b,c,d} \hat{f}_{a,b}(c) \hat{g}_{a,b}(d) \sigma_2(c,d) + 3\frac{\lambda_1}{d_1} \right)^2 \\
&\leq \left(\E_{a,b,c,d} \hat{f}_{a,b}(c) \hat{g}_{a,b}(d) \sigma_2(c,d)\right)^2 + 15\frac{\lambda_1}{d_1},
\end{align*}
using that $\E_{a,b,c,d} \hat{f}_{a,b}(c) \hat{g}_{a,b}(d) \sigma_2(c,d) \leq 1$ and $\lambda_1/d_1 \leq 1$. Now
\begin{align*}
    \left(\E_{a,b,c,d} \hat{f}_{a,b}(c) \hat{g}_{a,b}(d) \sigma_2(c,d)\right)^2 
= & \frac{1}{|V_1|^4|V_2|^4 }\left(\sum_{a,b} 1 \sum_{c,d}\hat{f}_{a,b}(c) \hat{g}_{a,b}(d) \sigma_2(c,d) \right)^2 \\
\leq & \frac{1}{|V_1|^2 } \sum_{a,b} \left(\frac{1}{|V_2|^2}\sum_{c,d}\hat{f}_{a,b}(c) \hat{g}_{a,b}(d) \sigma_2(c,d) \right)^2\\
 \leq & \frac{1}{|V_1|^2} \sum_{a,b} \left( \E_{c,d} \hat{f}_{a,b}(c) \hat{f}_{a,b}(d) + 3\frac{\lambda_2}{d_2} \right) \\
=& \E_{a,b} \left( \E_{c,d} f_1(a,c)f_1(b,c)f_1(a,d) f_1(b,d) + 3\frac{\lambda_2}{d_2}\right),
\end{align*}
by Cauchy-Schwarz and \eqref{eq3} respectively. As a consequence, we obtain
\begin{align*} |N(f_1, f_2, f_3, f_4)|^4&\le \E_{a,b,c,d}f_1(a,c)f_1(b,c)f_1(a,d) f_1(b,d) + 15\frac{\lambda_1}{d_1} + 3\frac{\lambda_2}{d_2}\\
&=M(f_1,f_1,f_1,f_1)+ O\left(\frac{\lambda_1}{d_1} + \frac{\lambda_2}{d_2}\right).\end{align*}
Notice that the same holds when $f_1$ on the right hand side is replaced by $f_i$ for $2\le i\le 4$. In short,
\[|N(f_1, f_2, f_3, f_4)|\le \min_{j}||f_j||_{\square (V_1\times V_2)}+O\left(\frac{\lambda^{1/4}}{d^{1/4}}\right).\]
This completes the proof. 
\end{proof}

With Lemmas \ref{co2.2} and \ref{co2.31} in hand, we are ready to prove Theorem \ref{thm:CS}.

\paragraph{Proof of Theorem \ref{thm:CS}:}
For any $\epsilon >0,$ by Lemma \ref{co2.2}, we can see that there exist $\sigma$-algebras $\mathcal B$ and $\mathcal C$ on $V_1$ and $V_2$, respectively, with complexity bounded above by $O\left(\epsilon^{-8}\right),$ so that
\begin{equation}\label{epBnd}
    ||S-\E(S| \B\vee\C)||_{\square(V_1\times V_2)}\le \epsilon.
\end{equation}
Let $g$ denote $\E(S| \B\vee\C),$ and define
$$h(x):=S(x) - g(x).$$
Therefore, \eqref{epBnd} gives us that 
\begin{equation}\label{hBnd}
    ||h||_{\square(V_1\times V_2)}\le \epsilon.
\end{equation}
Both $g$ and $h$ are functions from $V_1\times V_2$ to the interval $[-1,1].$
Recalling the definition of $N$ above, we see that
$$ N(S,S,S,S) = N(g,g,g,g)+N(h,h,h,h)+ R,$$
where $R$ is a sum over all expressions of the form
$$N(f_1, f_2, f_3, f_4),$$
where the $f_j$ in each term are either $g$ or $h$, but not all the same. Specifically, set $\Omega :=\{g,h\}^4\setminus\{(g,g,g,g),(h,h,h,h)\},$ denote a quadruple of functions by $F = (f_1, f_2, f_3, f_4)\in\Omega,$ and write
$$R = \sum_{F\in \Omega}\mathbb E_{a,b,c,d}f_1(a,c)f_2(a,d)f_3(b,c)f_4(b,d)E_1(a,b)E_2(c,d),$$
where $E_i(x,y)$ is the indicator that $xy\in E(G_i)$. Combining Lemma \ref{co2.31} and \eqref{hBnd} gives
$$|N(h,h,h,h)|\leq ||h||_{\square (V_1\times V_2)}+O\left(\frac{\lambda^{1/4}}{d^{1/4}}\right)\leq \epsilon+O\left(\frac{\lambda^{1/4}}{d^{1/4}}\right).$$
Similarly, for any other choice of $F\in \Omega,$ we must have $h$ in at least one entry, so we get
$$|N(F)|\leq \min_j||f_j||_{\square (V_1\times V_2)}+O\left(\frac{\lambda^{1/4}}{d^{1/4}}\right) \leq ||h||_{\square V_1\times V_2} +O\left(\frac{\lambda^{1/4}}{d^{1/4}}\right)\leq \epsilon+O\left(\frac{\lambda^{1/4}}{d^{1/4}}\right).$$
Putting these together we get that
\begin{equation}\label{NSgClose}
|N(S,S,S,S)-N(g,g,g,g)|=O\left(\epsilon+ \frac{\lambda^{1/4}}{d^{1/4}}\right)
\end{equation}
Similarly, by Lemma \ref{basic1}, we know that
$$M(F)\le \min_{i}||f_i||_{\square(V_1\times V_2)},$$
so we get that
\begin{equation}\label{MSgClose}
|M(S,S,S,S)-M(g,g,g,g)|=O\left(\epsilon\right).
\end{equation}

By definition, $g$ is a linear combination of indicator functions of atoms of the $\sigma$-algebra $\B \vee \C.$ By Lemma \ref{co2.2}, we know that there is some positive constant $c>0$ so that the number of terms in this linear combination is no more than $2^{c\epsilon^{-8}}.$ So we can write $N(g,g,g,g)$ as a linear combination of terms of the form
\begin{align*}
&N(B_1 \times C_1,B_2 \times C_2,B_3 \times C_3,B_4 \times C_4)\\
&= \mathbb E_{a,b,c,d}(B_1 \times C_1)(a,c)\cdot(B_2 \times C_2)(a,d)\cdot(B_3 \times C_3)(b,c)\cdot(B_4 \times C_4)(b,d)\sigma_1(a,b)\sigma_2(c,d),
\end{align*}
for some atoms $B_j\times C_j$ (and their indicator functions) in $\B \times \C$. Here as before we use $\sigma_i(x,y)$ equals $|V_i|/d_i$ if $\{x,y\}\in E_i$ and $0$ otherwise. However, if we split this up by variables, we get that
\begin{align*}
&N(B_1 \times C_1,B_2 \times C_2,B_3 \times C_3,B_4 \times C_4)\\
&= \mathbb E_{a,b,c,d}(B_1 \cap B_2)(a)\cdot(B_3 \cap B_4)(b)\cdot(C_1 \cap C_3)(c)\cdot(C_2 \cap C_4)(d)\sigma_1(a,b)\sigma_2(c,d)\\
&= \left(\mathbb E_{a,b}(B_1 \cap B_2)(a)\cdot(B_3 \cap B_4)(b)\sigma_1(a,b)\right)\left(\mathbb E_{c,d}(C_1 \cap C_3)(c)\cdot(C_2 \cap C_4)(d)\sigma_2(c,d)\right).
\end{align*}
By applying the expander mixing lemma as in the proof of Lemma \ref{co2.31}, we see
\begin{align*}
&N(B_1 \times C_1,B_2 \times C_2,B_3 \times C_3,B_4 \times C_4)\\
&= \left(\mathbb E_{a,b}(B_1 \cap B_2)(a)\cdot(B_3 \cap B_4)(b) + O\left(\frac{\lambda_1}{d_1}\right)\right)\left(\mathbb E_{c,d}(C_1 \cap C_3)(c)\cdot(C_2 \cap C_4)(d)+O\left(\frac{\lambda_2}{d_2}\right)\right)\\
&=M(B_1 \times C_1,B_2 \times C_2,B_3 \times C_3,B_4 \times C_4)+O\left(\frac{\lambda}{d}\right),
\end{align*}

where the last line uses the definition of $M$ and that each expectation is at most $1$. Since $g$ is a linear combination of at most $2^{c\epsilon^{-8}}$ terms, we see that
\[|N(g,g,g,g) -M(g,g,g,g)| = O\left( 2^{c'\epsilon^{-8}}\frac{\lambda}{d}\right), \]
for some positive constant $c'$. Using \eqref{NSgClose} followed by the previous estimate and \eqref{MSgClose}, we get that for some constant $k>0,$ we have
\begin{align*}
N(S,S,S,S) &\geq N(g,g,g,g) - k\epsilon -k\frac{\lambda^{1/4}}{d^{1/4}}\\
&\geq M(g,g,g,g) -k2^{c'\epsilon^{-8}}\frac{\lambda}{d}- k\epsilon -k\frac{\lambda^{1/4}}{d^{1/4}}\\
&\geq M(S,S,S,S) -k\epsilon- k2^{c\epsilon^{-8}}\frac{\lambda}{d}- k\epsilon -k\frac{\lambda^{1/4}}{d^{1/4}}.
\end{align*}
Now applying Proposition \ref{MCS} to this estimate gives us
\begin{equation}\label{punchline}
N(S,S,S,S) \geq \left(\frac{|S|}{|V_1||V_2|}\right)^4 -2k\epsilon- k2^{c\epsilon^{-8}}\frac{\lambda}{d} -k\frac{\lambda^{1/4}}{d^{1/4}}.
\end{equation}
Recall that by assumption, $|S|\geq \delta |V_1||V_2|,$ so to guarantee that $N = N(S,S,S,S)$ is positive, we just need to pick $\epsilon$ so that the right-hand-side of \eqref{punchline} is bigger than $\delta'^4$, or equivalently,
$$\delta^4 - \delta'^4 \geq 2k\epsilon +k2^{c\epsilon^{-8}}\frac{\lambda}{d} +k\frac{\lambda^{1/4}}{d^{1/4}}.$$

\section{Proof of Theorem \ref{cycle-main1}}
To prove Theorem \ref{cycle-main1}, we present two approaches based on two counting lemmas. 
While the second counting lemma is a direct consequence of the expander mixing lemma for a single graph, the first counting lemma is a stronger and more practical variant for tensor of two pseudo-random graphs, which is quite interesting on its own. 
\subsection{The first counting lemma for cycles}
Let us briefly describe the ideas of counting cycles here. Assume we want to count the number of cycles of length $2k$ for some integer $k\ge 2$. Given four vertices $x, y, z, w$, if $x$ and $y$ are connected by a path of length $k-1$, and the same happens for $z$ and $w$, then we will have a cycle of length $2k$ of the form $x-yw-zx$ (Figure 1) when there are edges between $x$ and $z$, and between $y$ and $w$. Thus, the problem is reduced to counting the number of pairs of edges between the endpoints of pairs of paths of length $k-1$.  

\begin{figure}[h!]
\begin{center}
\includegraphics[width=0.6\textwidth]{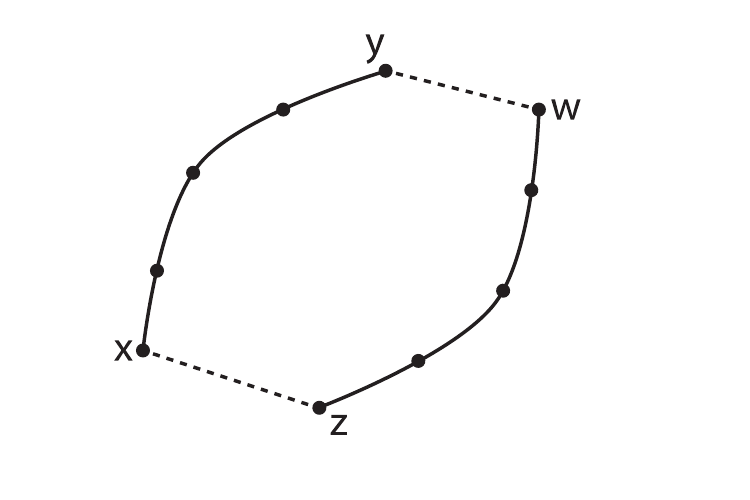}
\caption{Counting pairs of edges $xz$ and $yw$.}
\label{figure1}
\end{center}
\end{figure}
To this end, we make use of the notation of tensor of two pseudo-random graphs. For two graphs $\mathcal{G}_1=(V_1, E_1)$ and $\mathcal{G}_2=(V_2, E_2)$, the tensor product $\mathcal{G}_1\otimes \mathcal{G}_2$ is a graph with vertex set $V(\mathcal{G}_1\otimes \mathcal{G}_2)=V_1\times V_2$, and there is an edge between $(u,v)$ and $(u', v')$ if and only if $(u, u')\in E_1$ and $(v, v')\in E_2$. Suppose that the adjacency matrices of $\mathcal{G}_1$ and $\mathcal{G}_2$ are $A$ and $B$, respectively, then the adjacency matrix of $\mathcal{G}_1\otimes \mathcal{G}_2$ is the tensor product of $A$ and $B$. It is well-known that if $\gamma_1, \ldots, \gamma_n$ are eigenvalues of $A$ and $\gamma_1', \ldots, \gamma_m'$ are eigenvalues of $B$, then the eigenvalues of $A\otimes B$ are $\gamma_i\gamma_j'$ with $1\le i \le n$, $1\le j\le m$ (see \cite{merris1997multilinear} for more details).

It is not hard to use the expander mixing lemma to get the following.
\begin{proposition}\label{weakform}
Let $\mathcal{G}$ be an $(n, d, \lambda)$-graph. For two {non-negative} functions $f, g\colon V\times V\to \mathbb{R}$, we have
\[\left\vert \sum_{(x, z)\in E, (y, w)\in E}f(x, y)g(z, w)-\frac{d^2}{n^2}||f||_1||g||_1 \right\vert\le d\lambda||f||_2||g||_2.\]
\end{proposition}
Our first counting lemma offers better bounds as follows.
\begin{proposition}(First counting lemma)\label{keylemma}
Let $\mathcal{G}$ be an $(n, d, \lambda)$-graph. For two {non-negative} functions $f, g\colon V\times V\to \mathbb{R}$, we define $F(x)=\sum_{y}f(x, y)$, $G(z)=\sum_{w}g(z, w)$, $F'(y)=\sum_{x}f(x, y)$, and $G'(w)=\sum_{z}g(z, w)$. Then we have 
\[\left\vert \sum_{(x, z)\in E, (y, w)\in E}f(x, y)g(z, w)-\frac{d^2}{n^2}||f||_1||g||_1 \right\vert\le \lambda^2||f||_2||g||_2+\frac{d\lambda}{n^2}\left(||F||_2||G||_2+||F'||_2||G'||_2\right).\]
\end{proposition}
\begin{proof}
Suppose $\mathcal{G}$ is a $d$-regular graph on vertex set $V$ with $|V|=n$, and let $A$ denote its adjacency matrix.  For two real-valued functions $f, g\colon V\times V\to \mathbb{R}$, we define 
\[\langle f, g\rangle=\sum_{(v_1, v_2)\in V\times V}f(v_1, v_2){g}(v_1, v_2),\]
and 
\[||f||_2^2=\langle f, f\rangle.\]
We denote the set of all real-valued functions on $V\times V$ by $L^2(V\times V)$. {For the remainder of the proof we will assume that $f,g \in L^2(V\times V)$ are non-negative functions.}

We define 
\[A\otimes A f(v_1, v_2)=\sum_{(u_1, u_2)\colon (u_1, v_1)\in E, (u_2, v_2)\in E}f(u_1, u_2).\]
That is, $A\otimes A$ is the adjacency matrix of $\mathcal{G}\otimes \mathcal{G}$. In the remainder, we denote $A\otimes A$ by $B$.

Let $\lambda_1 \geq \lambda_2 \geq \cdots \geq \lambda_n$ be the eigenvalues of $A$ corresponding to eigenfunctions $e_1,\cdots, e_n$. Without loss of generality, assume that the $e_i$ form an orthonormal basis of $\mathbb{R}^n$. Then the eigenfunctions of $B$ are exactly $e_i \otimes e_j$ for all $1\leq i,j\leq n$ corresponding to eigenvalue $\lambda_{ij}:= \lambda_i \lambda_j$. 

We observe that
 \[f=\sum_{i,j}\langle f,e_i\otimes e_j\rangle e_i\otimes e_j.\]
 So 
 \[Bg=\sum_{i,j}\langle Bg,e_i\otimes e_j\rangle
e_i\otimes e_j=\sum_{e_i\otimes e_j}\lambda_{ij}\langle g,e_i\otimes e_j\rangle
e_i\otimes e_j\] 
We note that $A$ has a constant eigenfunction that will be denoted by $e_1$, i.e. 
\[e_1(v)=1/\sqrt{n}, ~\forall v\in V.\]
This means that $B$ also has constant eigenfunction defined by 
\[e_1\otimes e_1(u, v)=1/n\, ~\forall (u, v)\in V\times V.\]

We have 
\[
\sum_{(x, z)\in E, (y, w)\in E}f(x, y)g(z, w)=\langle f, Bg\rangle =\sum_{i,j}\lambda_{ij}\langle g, e_i\otimes e_j\rangle \langle f, e_i\otimes e_j\rangle.
\]
Define
\begin{align*}
    S_1:= &\lambda_{11} \langle g, e_1\otimes e_1\rangle \langle f, e_1\otimes e_1\rangle\\
    S_2:=\sum_{j=2}^n &\lambda_{1j}\langle g, e_1\otimes e_j\rangle \langle f, e_1\otimes e_j\rangle\\
    S_3:=\sum_{i=2}^n &\lambda_{i1} \langle g, e_i\otimes e_1\rangle \langle f, e_i\otimes e_1\rangle\\
    S_4:= \sum_{i,j=2}^n &\lambda_{ij}\langle g, e_i\otimes e_j\rangle \langle f, e_i\otimes e_j\rangle.
\end{align*}
And so
\[
\sum_{(x, z)\in E, (y, w)\in E}f(x, y)g(z, w) - S_1 = S_2 + S_3 + S_4.
\]

We now estimate each $S_i$. Since $\lambda_1 = d$ and $e_1$ is constant, it is easy to see that 
\[
S_1 = \lambda_{11} \left\langle f, \frac{1}{n}\mathbf{1}\right\rangle \left\langle g, \frac{1}{n}\mathbf{1}\right\rangle = \frac{d^2}{n^2} ||f||_1 ||g||_1.
\]

For $S_4$, if $i,j>1$ we have that $\lambda_{ij} \leq \lambda^2$ and hence
\begin{align*}
S_4 \leq \lambda^2 \sum_{i,j=2}^n \langle g, e_i\otimes e_j\rangle \langle f, e_i\otimes e_j\rangle & \leq \lambda^2 \left( \sum_{i,j=2}^n \langle g, e_i\otimes e_j \rangle^2\right)^{1/2}\left( \sum_{i,j=2}^n \langle f, e_i\otimes e_j \rangle^2\right)^{1/2}\\
&\leq \lambda^2 \left( \sum_{i,j=1}^n \langle g, e_i\otimes e_j \rangle^2\right)^{1/2}\left( \sum_{i,j=1}^n \langle f, e_i\otimes e_j \rangle^2\right)^{1/2}\\
&= \lambda^2 ||f||_2 ||g||_2,
\end{align*}
where the second inequality follows by Cauchy-Schwarz.

To estimate $S_2$, note that $\lambda_{1j} \leq \lambda d$, and 
\[
e_1\otimes e_j (v_1, v_2) = \frac{1}{\sqrt{n}} e_j(v_2).
\]
Using Cauchy-Schwarz, we have that 
\begin{align*}
S_2 \leq \lambda d \sum_{j=2}^n \langle g, e_1\otimes e_j \rangle \langle f, e_1\otimes e_j\rangle & \leq \lambda d \left(\sum_{j=2}^n \langle g, e_1\otimes e_j\rangle^2\right)^{1/2}\left(\sum_{j=2}^n \langle f, e_1\otimes e_j\rangle^2\right)^{1/2}\\
& \leq \lambda d \left(\sum_{j=1}^n \langle g, e_1\otimes e_j\rangle^2\right)^{1/2}\left(\sum_{j=1}^n \langle f, e_1\otimes e_j\rangle^2\right)^{1/2}
\end{align*}
To estimate this quantity, note that
\[
\langle g, e_1\otimes e_j\rangle  = \sum_{u,v} g(u,v) e_1\otimes e_j (u,v) = \frac{1}{\sqrt{n}} \sum_{u,v} g(u,v) e_j(v) = \frac{1}{\sqrt{n}} \sum_v G'(v) e_j(v),
\]
and similarly $\langle f, e_1\otimes e_j\rangle = \frac{1}{\sqrt{n}} \sum_v F'(v)e_j(v)$. Therefore, we have that 
\[
\sum_{j=1}^n \langle g, e_1\otimes e_j\rangle^2 = \sum_{j=1}^n \frac{1}{n}\sum_{u,v} G'(u)G'(v)e_j(u)e_j(v) = \frac{1}{n} \sum_{u,v} \left( G'(u)G'(v) \sum_{j=1}^n e_j(u)e_j(v)\right).
\]
Now notice that because the $e_i$ form an orthonormal basis, we have that 
\[
\sum_{j=1}^n e_j(u)e_j(v) = \begin{cases}
1 & u=v\\
0 & u\not=v.
\end{cases}
\]
Hence we have 
\[
\sum_{j=1}^n \langle g, e_1 \otimes e_j\rangle^2 = \frac{1}{n} \sum_{u=1}^n \left((G'(u))^2 \sum_{j=1}^n e_j(u)^2\right) = \frac{1}{n}\sum_{u=1}^n (G'(u))^2 = \frac{1}{n}||G'||_2^2.
\]
Similarly $\sum_{j=1}^n \langle f, e_1 \otimes e_j\rangle^2 = \frac{1}{n} ||F'||_2^2$. Combining everything we have that 
\[
S_2 \leq \frac{\lambda d}{n^2} ||G'||_2 ||F'||_2.
\]
A symmetric proof shows that 
\[
S_3 \leq \frac{\lambda d}{n^2} ||G||_2 ||F||_2.
\]
\end{proof}
\subsection{The second counting lemma for cycles}
Assume we want to count the number of cycles of length $2k$ for some integer $k\ge 1$. Our second strategy for cycles can be explained as follows. Given three vertices $x, y,$ and $z$, if $x$ and $y$ are connected by a path of length $k$, and $x$ and $z$ are connected by a path of length $k-1$, then we have a cycle of length $2k$ of the form $x-yz-x$ if and only if $y$ and $z$ are adjacent (Figure 2). So the problem is reduced to counting the number of edges between the endpoints of pairs of paths pinned at a vertex. 
\begin{figure}[h!]
\begin{center}
\includegraphics[width=0.6\textwidth]{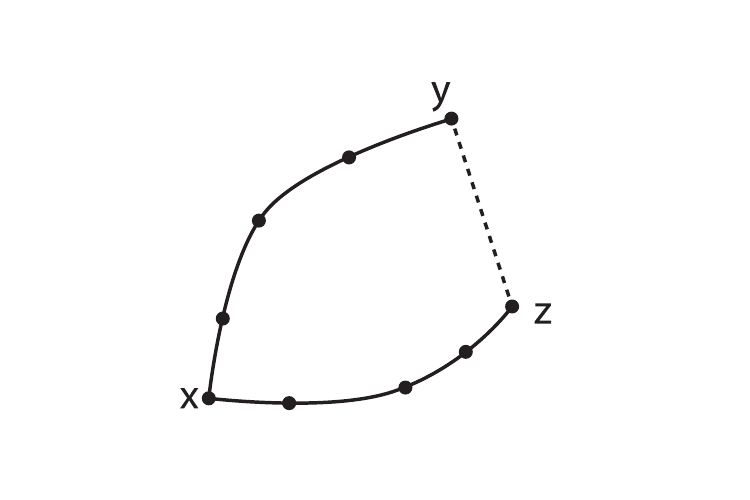}
\caption{Counting edges $yz$.}
\label{figure2}
\end{center}
\end{figure}
\begin{proposition}(Second counting lemma)\label{prop-cycle1}
Let $\mathcal{G}$ be an $(n, d, \lambda)$-graph. Let $U$ be a set of vertices in $\mathcal{G}$. For any two vertices $x$ and $y$, let $p_k(x, y)$ be the number of paths of length $k$ between $x$ and $y$ with vertices in between belonging to $U$. Then we have   
\[\left\vert C_{2k+1}(U)-\frac{d}{n}\sum_{x\in U}\left(\sum_{y\in U}p_k(x, y)\right)^2\right\vert\le \lambda \sum_{x, y\in U}p_k(x, y)^2,\]
and 
\[\left\vert C_{2k}(U)-\frac{d}{n}\sum_{x\in U}\left(\sum_{y\in U}p_k(x, y)\right)\cdot \left(\sum_{z\in U}p_{k-1}(x, z)\right)\right\vert\le \lambda \sum_{x\in U}\left(\sum_{y\in U}p_k(x, y)^2\right)^{1/2}\cdot \left(\sum_{z\in U}p_{k-1}(x, z)^2\right)^{1/2}.\]
\end{proposition}
\begin{proof}
We first observe that the number of odd cycles of length $2k+1$ in $U$ is equal to the sum 
\[\sum_{x, y, z\in U^3, (y, z)\in E(G)}p_k(x, y)p_k(x, z).\]
Given $x\in U$, set $f(y)=U(y)p_k(x, y)$, then the above sum can be rewritten as 
\[\sum_{x\in U}\sum_{(y, z)\in E(G)}f(y)f(z).\]
Applying Lemma \ref{th:expanderMixing}, the first statement is proved.

For the second statement, as above, the number of even cycles of length $2k$ in $U$ is equal to the sum 
\[\sum_{x, y, z\in U^3, (y, z)\in E(G)}p_k(x, y)p_{k-1}(x, z).\]
Given $x\in U$, set $f(y)=U(y)p_k(x, y)$ and $g(z)=U(z)p_{k-1}(x, z)$, then the above sum can be rewritten as 
\[\sum_{x\in U}\sum_{(y, z)\in E(G)}f(y)g(z).\]
Applying Lemma \ref{th:expanderMixing}, the proposition is proved.
\end{proof}
Using the facts that 
\[\sum_{x\in U}\left(\sum_{y\in U}p_k(x, y)\right)^2=P_{2k}(U),\]
\[\sum_{x, y\in U}p_{k}(x, y)^2=C_{2k}(U),\]
\[\sum_{x\in U}\left(\sum_{y\in U}p_k(x, y)\right)\cdot \left(\sum_{z\in U}p_{k-1}(x, z)\right)=P_{2k-1}(U),\]
and the following application of Cauchy-Schwarz,
\[\sum_{x\in U}\left(\sum_{y\in U}p_k(x, y)^2\right)^{1/2}\cdot \left(\sum_{z\in U}p_{k-1}(x, z)^2\right)^{1/2}\le \left( C_{2k}(U)C_{2k-2}(U)\right)^{1/2},\]
one derives the following corollary. 
\begin{corollary}\label{cor:5.2-odd}
Let $\mathcal{G}$ be an $(n, d, \lambda)$-graph. Let $U$ be a set of vertices in $\mathcal{G}$. Then 
\[\left\vert C_{2k+1}(U)-\frac{d}{n}P_{2k}(U) \right\vert\le \lambda C_{2k}(U),\]
and 
\[\left\vert C_{2k}(U)-\frac{d}{n}P_{2k-1}(U) \right\vert\le \lambda \left( C_{2k}(U)C_{2k-2}(U)\right)^{1/2}.\]
\end{corollary}
\subsection{Distribution of paths}
We have seen that to apply the two counting lemmas, we need to have estimates on the paths of a given length in a vertex set. We now provide relevant results on paths.
\begin{proposition}\label{paths}
Let $\mathcal{G}$ be an $(n, d, \lambda)$-graph, $k\ge 1$ an integer,  and $U$ be a vertex set with $\lambda \cdot \frac{n}{d} = o(|U|)$. Let $P_k(U)$ denotes the number of paths of length $k$ in $U$. Then we have 
\[P_k(U)=\left[1 + \Theta\left( \frac{\lambda n}{d|U|}\right) \right]\frac{|U|^{k+1}d^k}{n^k}.\]
\end{proposition}

\begin{proof}
We first prove the following two estimates: 
\begin{equation}\label{eq:odd}\left\vert P_{2k+1}(U)-\frac{dP_k(U)^2}{n}\right\vert\le \lambda P_{2k}(U), \end{equation}and 
\begin{equation}\label{even}~~\left\vert P_{2k}(U)-\frac{dP_k(U)P_{k-1}(U)}{n}\right\vert\le \lambda\sqrt{P_{2k}(U)P_{2k-2}(U)}.\end{equation}
For $u\in U$, let $f(u)$ be the number of paths of length $k$ of the form $(u_1, \ldots, u_k, u)$ where $u_i\in U$. Similarly, for $v\in U$, let $g(v)$ be the number of paths of length $k$ of the form $(v_1, \ldots, v_k, v)$ where $v_i\in U$. To use Lemma \ref{th:expanderMixing} we need to estimate the norms and the inner product. We have that the adjacency matrix $A$ acts on $g$ by the formula
\[
Ag(u)=\sum_{(u,v)\in E(\mathcal{G})}g(v)
\]
which is the number of paths of length $k+1$ of the form $(v_1,\dots,v_k,v,u)$. For the inner product, we have
\[
\langle f, Ag \rangle= \sum_{u\in V(\mathcal{G})}f(u){Ag}(u)=\sum_{u\in V(\mathcal{G})}f(u)Ag(u)=P_{2k+1}(U).
\]

It is clear that 
\[
   \mathbb{E}(f)=\mathbb{E}(g)=\frac{1}{|V|}\cdot P_{k}(U).
\]
and
\[||f||_2^2=||g||_2^2=P_{2k}(U).\]
Applying Lemma \ref{th:expanderMixing} we have that 
\[
\left| P_{2k+1}(U)-d|V|\left(\frac{1}{|V|}\cdot P_{k}(U) \right)^{2} \right|\leq
\lambda P_{2k}(U)\]
which is equivalent to \eqref{eq:odd}. The estimate \eqref{even} also follows from a similar argument with the same $f$ and $g(v)$ defined to be the number of paths of length $k-1$ of the form $(v_1,\dots,v_{k-1},v)$. 

We now proceed by induction on $k$. The case $k=0$ is trivial and the case $k=1$ follows from Lemma \ref{th:expanderMixing} and the estimate (\ref{even}). 

Suppose that the statement holds for all $2k\ge 1$. We now show that it also holds for $2k+1$ and $2k+2$. Indeed, 
it follows from the estimate (\ref{eq:odd}) and induction hypothesis that we have
\begin{align*}
P_{2k+1}(U)&\le \frac{d}{n}P_k(U)^2+\lambda P_{2k}(U) \\&\le 
\frac{d}{n}\frac{|U|^{2k+2}d^{2k}}{n^{2k}}\left(1 + O\left(\frac{\lambda n}{d |U|}\right)\right)^2 + \lambda \frac{|U|^{2k+1} d^{2k}}{n^{2k}}\left(1 + O\left(\frac{\lambda n}{d |U|}\right)\right) \\&=
\frac{|U|^{2k+2}d^{2k+1}}{n^{2k+1}} \left(1 + O\left(\frac{\lambda n}{d|U|}\right)\right)
\end{align*}
whenever $|U|\gg \lambda \frac{n}{d}$. The lower bound follows in the same way.

For the case $2k+2$, it also follows from the estimate (\ref{even}) that 
\[P_{2k+2}(U)\le \frac{dP_{k}(U)P_{k+1}(U)}{n}+\lambda \sqrt{P_{2k}(U)P_{2k+2}(U)}.\]
Solving this inequality in $x=\sqrt{P_{2k+2}(U)}$, we obtain 
\[
P_{2k+2}(U) \leq \left(  \frac{-\lambda\sqrt{P_{2k}(U)} + \sqrt{\lambda^2 P_{2k}(U) + \frac{4dP_k(U)P_{k+1}(U)}{n}}}{2}\right)^2.
\]
Using the induction hypothesis and that $\frac{\lambda n}{d}=o(|U|),$ we have that
$$\lambda^2P_{2k}(U) = o\left(\frac{dP_k(U)P_{k+1}(U)}{n}\cdot \frac{\lambda n}{d|U|} \right)$$
and that 
\[
\lambda\sqrt{P_{2k}(U)} \sqrt{\frac{dP_k(U)P_{k+1}(U)}{n}} = O \left( \frac{|U|^{2k+3}d^{2k+2}}{n^{2k+2}} \cdot \frac{\lambda n}{d|U|}\right).
\]
Hence the entire expression is bounded above by 
\[
\frac{|U|^{2k+3}d^{2k+2}}{n^{2k+2}}\left[1 + O\left(\frac{\lambda n}{d|U|} \right)\right].
\]
Using lower bounds of the estimates (\ref{eq:odd}) and (\ref{even}), and an identical argument also gives us
\[P_{k}(U)\ge \left[1-O\left(\frac{\lambda n}{d |U|}\right)\right]|U|^{k+1} \left(\frac{d}{n}\right)^{k},\]
under the condition $\lambda \frac{n}{d}=o(|U|)$. This completes the proof of the proposition. 
\end{proof}

\subsection{Proof of Theorem \ref{cycle-main1} using the first counting lemma}
\begin{proof}[Proof of Theorem \ref{cycle-main1}]
We proceed by induction on $m$. 

We first start with the base case $m=4$. 

Let $f,g: U\times U \to \mathbb{R}$ defined by 
\[
f(x,y)=\begin{cases}
1,~ \text{if } (x,y)\in E(\mathcal{G})\\
0,~ \text{otherwise}
\end{cases}
\]
and
\[
g(z,w)=\begin{cases}
1,~ \text{if } (z,w)\in E(\mathcal{G})\\
0,~ \text{otherwise.}
\end{cases}
\]
It is clear that $C_4(U)=\sum_{(x,z), (y,w)\in E(\mathcal{G})}f(x, y)g(z, w)$. To apply Proposition \ref{keylemma}, we need to check the norms of functions $f$, $g$, $F$, $G$, $F'$, and $G'$.

Using Proposition \ref{paths}, we have
\begin{align*}
||f||_1&=\sum_{x,y \in U} f(x,y)=P_1(U)=\left(1+\Theta\left(\frac{\lambda n}{d|U|}\right)\right)|U|^2\frac{d}{n},\\
||f||_2&=\left(\sum_{x,y \in U} f(x,y)^2 \right)^{1/2}=\left(\sum_{x,y \in U} f(x,y) \right)^{1/2}=\left(P_1(U) \right)^{1/2}=\left(1+\Theta\left(\frac{\lambda n}{d|U|}\right)\right)|U|\sqrt{ \frac{d}{n}},
\end{align*}
using the Taylor series for $\sqrt{1+x}$ and the assumption that $ \frac{\lambda n}{d} = o(|U|)$.
Similarly,
\begin{align*}
    ||g||_1 =\left(1+\Theta\left(\frac{\lambda n}{d|U|}\right)\right) |U|^{2}\frac{d}{n},~ \text{ and } ||g||_2 =\left(1+\Theta\left(\frac{\lambda n}{d|U|}\right)\right) |U|\sqrt{ \frac{d}{n}}.
\end{align*}
For functions $F, G, F', G'$ defined as in Proposition \ref{keylemma}, we have that
\[
||F||_2=\left(\sum_{x \in U} F(x)^2 \right)^{1/2}=(P_2(U))^{1/2}=\left( \frac{|U|^3 d^2}{n^2} \left(1+\Theta\left(\frac{\lambda n}{d|U|}\right)\right)\right)^{1/2}.
\]
Similarly,
\[
||G||_2=||F'||_2=||G'||_2=\left( \frac{|U|^3 d^2}{n^2} \left(1+\Theta\left(\frac{\lambda n}{d|U|}\right)\right)\right)^{1/2}.
\]
Substituting these estimates into Proposition \ref{keylemma}, we have that 
\[
\left|C_4(U)-\left(1+\Theta\left(\frac{\lambda n}{d|U|}\right)\right)|U|^4\frac{d^4}{n^4} \right| \leq \frac{\lambda^2|U|^2 d}{n}\left( 1 + O\left( \frac{\lambda n}{d|U|}\right)\right) + \frac{2\lambda |U|^3 d^3}{n^4}\left( 1 + O\left( \frac{\lambda n}{d|U|}\right)\right).
\]
Using the assumption that $\frac{\lambda n}{d} = o(|U|)$ completes this case.

Assume that the statement holds for any cycle of length smaller than $m-1$, we now show that it holds for cycles of length $m$. 

We fall into two cases:

{\bf Case $1$: $m=2k+1$.}

As above, for $x, y\in U$, we define 
\[f(x, y)=\text{the number of paths of length $k$ between $x$ and $y$},\]
and 
\[g(x, y)=\text{the number of paths of length $k-1$ between $x$ and $y$}.\]
Then
\begin{align*}
    &||f||_1=P_{k}(U)=\left(1+\Theta\left(\frac{\lambda n}{d|U|} \right)\right)\frac{d^{k}}{n^{k}}|U|^{k+1},\\
    &||f||_2^2=C_{2k}(U) = (1+o(1))\frac{|U|^{2k}d^{2k}}{n^{2k}} + O\left(\frac{\lambda^{2k-2}|U|^2 d}{n} \right),\\
    &||F||_2^2=||F'||_2^2=P_{2k}(U) = (1+o(1))\frac{d^{2k}}{n^{2k}}|U|^{2k+1},
\end{align*}
and 

\begin{align*}
    &||g||_1=P_{k-1}(U)=\left(1+\Theta\left(\frac{\lambda n}{d|U|} \right)\right)\frac{d^{k-1}}{n^{k-1}}|U|^k, \\
    &||g||_2^2=C_{2}(U) = (1+o(1)) \frac{|U|^{2}d}{n} \quad \mbox{if}\,\, m=5\\
     &||g||_2^2=C_{2k-2}(U) = (1+o(1)) \frac{|U|^{2k-2}d^{2k-2}}{n^{2k-2}} + O\left(\frac{\lambda^{2k-4}|U|^2 d}{n} \right) \quad \mbox{if}\,\, m\geq 7\\
    &||G||_2^2=||G'||_2^2=P_{2k-2}(U)=(1+o(1))\frac{d^{2k-2}}{n^{2k-2}}|U|^{2k-1}
\end{align*}

Applying Proposition \ref{keylemma}, 
\[
\left| C_{2k+1}(U) - \frac{d^2}{n^2}P_k(U)P_{k-1}(U)\right| \leq \lambda^2 \sqrt{C_{2k}(U)C_{2k-2}(U)} + 2\frac{d\lambda}{n^2}(\sqrt{P_{2k}(U)P_{2k-2}(U)})
\]
When $m=5$, we have 
\[
\left|C_5(U) - \left(1 + \Theta\left(\frac{\lambda n}{d|U|}\right)\right)\frac{|U|^5 d^5}{n^5}\right| = O\left(\lambda^2 \sqrt{\frac{|U|^6d^5}{n^5} + \frac{\lambda^2|U|^4 d^2}{n^2} } + \frac{\lambda d^4 |U|^4}{n^5}\right).
\]
Note that $\frac{\lambda d^4 |U|^4}{n^5} = o \left(\frac{\lambda n}{d|U|} \frac{|U|^5d^5}{n^5}\right)$. If $|U| \leq \frac{\lambda n^{3/2}}{d^{3/2}}$, then the second term in the square root is bigger than the first, and hence
\[
\lambda^2 \sqrt{\frac{|U|^6d^5}{n^5} + \frac{\lambda^2|U|^4 d^2}{n^2} } = O\left(\frac{\lambda^3 |U|^2 d}{n} \right).
\]
If $|U| \geq \frac{\lambda n^{3/2}}{d^{3/2}}$ then the first term is bigger than the second and we have 
\[
\lambda^2 \sqrt{\frac{|U|^6d^5}{n^5} + \frac{\lambda^2n^4 d^2}{n^2} } = O\left(\frac{\lambda^2 |U|^3 d^{5/2}}{n^{5/2}} \right).
\]
Using the assumption that $|U| \geq \frac{\lambda n^{3/2}}{d^{3/2}}$ gives that $\frac{\lambda^2 |U|^3 d^{5/2}}{n^{5/2}} \leq \frac{\lambda |U|^4 d^4}{n^4} = \frac{\lambda n}{d|U|} \frac{|U|^5d^5}{n^5}$. In either case the inequality is satisfied. 

For $m \geq 7$ we have 
\begin{align*}
&\left|C_{2k+1}(U) - \left(1 + \Theta\left(\frac{\lambda n}{d|U|}\right)\right)\frac{|U|^{2k+1} d^{2k+1}}{n^{2k+1}}\right| = \\ &O\left(\lambda^2 \sqrt{\left[\frac{|U|^{2k}d^{2k}}{n^{2k}} + \frac{\lambda^{2k-2}|U|^2d}{n} \right]\left[\frac{|U|^{2k-2}d^{2k-2}}{n^{2k-2}} + \frac{\lambda^{2k-4}|U|^2d}{n} \right]} + \frac{\lambda |U|^{2k} d^{2k}}{n^{2k+1}}\right).
\end{align*}
First note that $\frac{\lambda |U|^{2k} d^{2k}}{n^{2k+1}} = o\left( \frac{\lambda n}{d} \frac{|U|^{2k+1} d^{2k+1}}{n^{2k+1}}\right)$ so we may ignore this term. Hence if each of the four terms 
\[
\frac{|U|^{4k-2} d^{4k-2}}{n^{4k-2}},\quad \frac{\lambda^{2k-2}|U|^{2k}d^{2k-1}}{n^{2k-1}}, \quad \frac{\lambda^{2k-4}|U|^{2k+2}d^{2k+1}}{n^{2k+1}},\quad \frac{\lambda^{4k-6}|U|^4d^2}{n^2},
\]
is either 
\[
O\left( \frac{|U|^{4k}d^{4k}}{\lambda^2 n^{4k}}\right) \quad \mbox{or} \quad O\left( \frac{\lambda^{4k-6}|U|^4 d^2}{n^2}\right),
\]
then we are done. The fourth term trivially satisfies the inequality. The first term is $o\left( \frac{|U|^{4k}d^{4k}}{\lambda^2 n^{4k}}\right)$ and 
\[
\frac{\lambda^{2k-2}|U|^{2k}d^{2k-1}}{n^{2k-1}} = o\left( \frac{\lambda^{2k-4}|U|^{2k+2}d^{2k+1}}{n^{2k+1}}\right)
\]
by the assumption that $\frac{\lambda n}{d} = o(|U|)$. 
Finally, if $|U| \geq \lambda \left(\frac{n}{d}\right)^{(2k-1)/(2k-2)}$ then 
\[
\frac{\lambda^{2k-4}|U|^{2k+2}d^{2k+1}}{n^{2k+1}} \leq \frac{|U|^{4k}d^{4k}}{\lambda^2 n^{4k}}. 
\]
Otherwise 
\[
\frac{\lambda^{2k-4}|U|^{2k+2}d^{2k+1}}{n^{2k+1}} \leq\frac{\lambda^{4k-6}|U|^4 d^2}{n^2}.
\]
\medskip

{\bf Case $2$: $m=2k$.}

For this case, we want to apply Proposition \ref{keylemma} again, so we need to define suitable functions $f$ and $g$, namely, for $x, y\in U$, \[f(x, y)=g(x, y)=\text{the number of paths of length $k-1$ between $x$ and $y$}.\]

Then, by inductive hypothesis and Proposition \ref{paths}, one has
\begin{align*}
    &||f||_1=||g||_1=P_{k-1}(U)=\left( 1 + \Theta\left( \frac{\lambda n}{d|U|}\right)\right)\frac{d^{k-1}}{n^{k-1}}|U|^k,\\
    &||f||_2^2=||g||_2^2=C_{2k-2}(U) =  (1 + o(1))\frac{|U|^{2k-2}d^{2k-2}}{n^{2k-2}}+\Theta\left(\lambda^{2k-4}\frac{d}{n}|U|^2\right),\\
    &||F||_2^2=||G||_2^2=||F'||_2^2=||G'||_2^2=P_{2k-2}(U)=(1+o(1))\frac{d^{2k-2}}{n^{2k-2}}|U|^{2k-1}.
\end{align*}
By applying Proposition \ref{keylemma} and the estimates above, we get that
\[
\left| C_{2k}(U) - \frac{d^2}{n^2} P_{k-1}(U)\right| \leq \lambda^2 C_{2k-2}(U) + 2\frac{\lambda d}{n^2}P_{2k-2}(U),
\]
and hence
\[
\left|C_{2k}(U) - \left(1 + \Theta\left(\frac{\lambda n}{d|U|} \right) \right)\frac{|U|^{2k}d^{2k}}{n^{2k}} \right| =  O\left(\frac{|U|^{2k-2}d^{2k-2}\lambda^2}{n^{2k-2}} + \frac{\lambda^{2k-2} |U|^2 d}{n}  +  \frac{\lambda |U|^{2k-1} d^{2k-1}}{n^{2k}}\right).
\]
Since $\frac{\lambda |U|^{2k-1} d^{2k-1}}{n^{2k}}$ and $\frac{|U|^{2k-2}d^{2k-2}\lambda^2}{n^{2k-2}}$ are both $o\left(\frac{\lambda n}{d} \frac{|U|^{2k}d^{2k}}{n^{2k}} \right)$ (the latter because of the assumption that $\frac{\lambda n}{d} = o(|U|)$), we are done.

\end{proof}
\subsection{Proof of Theorem \ref{cycle-main1} using the second counting lemma}
\begin{proof}[Proof of Theorem \ref{cycle-main1}]
Using the second counting lemma, we are able to prove Theorem \ref{cycle-main1} for all $m\ge 5$, i.e.
\[C_m(U)=  \frac{|U|^md^m}{n^m}+ \Theta\left( \frac{\lambda |U|^{m-1}d^{m-1}}{n^{m-1}}+ \lambda^{m-2}\frac{d}{n}|U|^2\right),\]
but for $m=4$, the result becomes slightly weaker, namely,
{
\[
C_4(U) = O\left( \frac{|U|^4 d^4}{n^4} + \frac{\lambda^2 |U|^2 d}{n}\right).
\]
}
We proceed by induction. 
{\bf Case $1$: $m=2k$.}

For $m=4$, by Corollary \ref{cor:5.2-odd} and Proposition \ref{paths}, we have that 
\begin{equation}\label{eqc4}
\left| C_4(U) - \frac{|U|^4 d^4}{n^4}\left( 1 + \Theta\left( \frac{\lambda n}{d|U|}\right) \right) \right| \leq \sqrt{(1+o(1))\frac{|U|^2\lambda^2 d}{n}C_4(U)},
\end{equation}
using that $C_2(U) = P_2(U)$ and $P_2(U) = \frac{d |U|^2}{n}(1+o(1))$ by the assumption that $\frac{\lambda n}{d} = o(|U|)$. Hence we may set up a quadratic in $\sqrt{C_4(U)}$ to obtain 
\begin{align*}
C_4(U) &\leq \left( \frac{\sqrt{(1+o(1))\frac{|U|^2\lambda^2 d}{n}} + \sqrt{(1+o(1))\frac{|U|^2\lambda^2 d}{n} + 4\frac{|U|^4 d^4}{n^4}\left( 1 + \Theta\left( \frac{\lambda n}{d|U|}\right) \right) }}{2} \right)^2\\
& = O\left( \frac{|U|^4 d^4}{n^4} + \frac{\lambda^2 |U|^2 d}{n}\right).
\end{align*}
This gives the desired estimate for $C_4$. If one wishes to have the main term $\frac{|U|^4d^4}{n^4}$ instead of $c\frac{|U|^4d^4}{n^4}$, for some positive constant $c$, with this approach, then it can be pushed further as follows. 

{
Using the above upper bound for $C_4$ and the estimate (\ref{eqc4}) gives us
\[
\left\vert C_4(U) - \frac{|U|^4 d^4}{n^4}\left( 1 + \Theta\left( \frac{\lambda n}{d|U|}\right) \right)\right\vert = O\left(\sqrt{\frac{|U|^4\lambda^4 d^2}{n^2} + \frac{|U|^6\lambda^2 d^5}{n^5}}\right).
\]
This gives 
\[C_4(U)=\frac{|U|^4d^4}{n^4}+\Theta\left(\frac{\lambda^2d|U|^2}{n}+\frac{\lambda |U|^3d^3}{n^3}+  \frac{|U|^3\lambda d^{5/2}}{n^{5/2}}\right).\]


}

Note that this gives the estimate \eqref{cycle} under the more restrictive condition that $\lambda \frac{n^{3/2}}{d^{3/2}} = o(|U|)$.

Assume the upper bound holds for all cycles of length at most $m-1$. We now show that it also holds for cycles of length $m$. Indeed, if $m=2k$, then we can apply Corollary \ref{cor:5.2-odd} to have 
\[C_{2k}(U)\le \frac{d}{n}P_{2k-1}(U)+\lambda \left( C_{2k}(U)C_{2k-2}(U)\right)^{1/2}.\]

Solving a quadratic in $\sqrt{C_{2k}(U)}$ gives 
\[
C_{2k}(U) \leq \left(\frac{ \lambda\sqrt{C_{2k-2}(U)} + \sqrt{\lambda^2 C_{2k-2}(U) + 4\frac{d}{n}P_{2k-1}(U)}}{2} \right)^2 
\]
Using Proposition \ref{paths} gives that 
\[
C_{2k}(U) - \frac{|U|^{2k}d^{2k}}{n^{2k}}\left(1 + \Theta\left( \frac{\lambda n}{d||U|}\right) \right)  = O\left(\lambda^2 C_{2k-2}(U) + \lambda\sqrt{\lambda^2 (C_{2k-2}(U))^2 + \frac{d}{n}C_{2k-2}(U)P_{2k-1}(U)} \right).
\]
By the inductive hypothesis and the assumption that $\frac{\lambda n}{d} = o(|U|)$, we have that 
\[
\lambda^2C_{2k-2}(U) = O\left(\frac{\lambda |U|^{2k-1}d^{2k-1}}{n^{2k-1}} + \frac{\lambda^{2k-2}|U|^2d}{n} \right),
\]
and hence we are done as long as 
\[
\lambda\sqrt{\frac{d}{n}C_{2k-2}(U)P_{2k-1}(U)}  = O\left(\frac{\lambda |U|^{2k-1}d^{2k-1}}{n^{2k-1}} + \frac{\lambda^{2k-2}|U|^2d}{n} \right).
\]
By the inductive hypothesis and Proposition \ref{paths}, we have 
\[
\lambda\sqrt{\frac{d}{n}C_{2k-2}(U)P_{2k-1}(U)} = O\left(\sqrt{ \frac{\lambda^2 |U|^{4k-2}d^{4k-2}}{n^{4k-2}} + \frac{\lambda^3 |U|^{4k-3}d^{4k-3}}{n^{4k-3}} + \frac{\lambda^{2k-2}|U|^{2k+2}d^{2k+1}}{n^{2k+1}}} \right).
\]
Since $\frac{\lambda n}{d} = o(|U|)$ we have that $\frac{\lambda^3 |U|^{4k-3}d^{4k-3}}{n^{4k-3}}=o\left(\frac{\lambda^2 |U|^{4k-2}d^{4k-2}}{n^{4k-2}}\right)$. Therefore, because 
\[
\sqrt{\frac{\lambda^2 |U|^{4k-2}d^{4k-2}}{n^{4k-2}}} = \frac{\lambda |U|^{2k-1} d^{2k-1}}{n^{2k-1}},
\]
we are done as long as $\frac{\lambda^{k-1} |U|^{k+1} d^{(2k+1)/2}}{n^{(2k+1)/2}}$ is small enough. If $|U| \geq \frac{\lambda n^{(2k-3)/(2k-4)}}{d^{(2k-3)/(2k-4)}}$ then 
\[
\frac{\lambda^{k-1} |U|^{k+1} d^{(2k+1)/2}}{n^{(2k+1)/2}} \leq \frac{\lambda |U|^{2k-1}d^{2k-1}}{n^{2k-1}}.
\]
Otherwise, we have 
\[
\frac{\lambda^{k-1} |U|^{k+1} d^{(2k+1)/2}}{n^{(2k+1)/2}} \leq \frac{\lambda^{2k-2}|U|^2d}{n},
\]
and the upper bound is complete. An analogous calculation gives the corresponding lower bound and we omit the details.

{\bf Case 2: $m=2k+1$.}

This case follows directly from Corollary \ref{cor:5.2-odd} and the case $m=2k$ above.
\end{proof}
\section{Proofs of Theorem \ref{tree-thm1} and Theorem \ref{tree-thm12}}
\subsection{Technical lemmas}

To prove Theorems \ref{tree-thm1} and \ref{tree-thm12}, we use the following results, which are direct consequences of the expander mixing lemma. The first result guarantees that vertex sets bigger than $\lambda n / d$ will have an edge of each color.

\begin{lemma}\label{lm1}
Let $G$ be an $(n, d, \lambda)$-colored graph with color set $D$, and $A, B \subset V(G)$ with $|A|=|B| > \frac{\lambda n}{d}$. Then for each color $c\in D$, there exists an edge $uv$ of color $c$ with $u\in A$ and $v \in B$. In other words, every vertex set of size greater than $\frac{\lambda n}{d}$ determines every color. 
\end{lemma}

\begin{proof}
For each color $c$ in $D$, let $G_c$ be the induced graph on $c$, then $G_c$ is an $(n,d,\lambda)$-graph. Applying Lemma \ref{th:expanderMixing} with 
\[
f(u)=\begin{cases}
1,~ \text{if } u\in A\\
0,~ \text{otherwise}
\end{cases}
\]
and
\[
g(v)=\begin{cases}
1,~ \text{if } v\in B\\
0,~ \text{otherwise},
\end{cases}
\]
we have 
\begin{align*}
    \langle f,Ag\rangle= e(A,B)\coloneqq  \left|\{(a,b)\in A\times B:ab\in E(G)\}\right|.
\end{align*}
It is clear that 
\[
\mathbb{E}(f)=\frac{|A|}{n},~ \mathbb{E}(g)=\frac{|B|}{n} 
\]
and
\[
\|f\|_2=\sqrt{|A|},~\|g\|_2=\sqrt{|B|}.
\]
Then we have
\[
\left|e(A,B)-\frac{d}{n}|A||B| \right|\leq \lambda\sqrt{|A||B|}.
\]
So
\[
    e(A,B)\geq \frac{d}{n}|A||B|-\lambda\sqrt{|A||B|}. 
\]
Since $|A|=|B| > \frac{\lambda n}{d}$,
\begin{align*}
    e(A,B)&\geq \frac{d}{n}|A|^2-\lambda|A|\geq|A|\left(\frac{d}{n}\cdot|A|-\lambda \right)> |A|\left(\frac{d}{n}\cdot\frac{\lambda n}{d}-\lambda  \right)> 0.
\end{align*}
Which means that there exists at least one edge of color $c$ between $A$ and $B$.  
\end{proof}


The next technical lemma uses the previous result to give an upper bound on the number of vertices with small degree of a given edge color.

\begin{lemma}\label{main4}
Let $G$ be an $(n, d, \lambda)$-colored graph with color set $D$, and let $U\subset V(G)$ with $|U|=r\cdot \frac{\lambda n}{d}$. Then for any fixed color $d\in D$, $s \in \mathbb{N}$, there are at most $s\cdot \frac{\lambda n}{d}$ vertices of $U$ for which  each of them is incident with less than $s$ edges colored by $d$. 
\end{lemma}

\begin{proof}
Let $H$ be induced graph on color $d$. Consider the subgraph $H^*$ of $H$ generated by only those vertices of degree less than $s$, so $H^*$ can be $s$-colorable. That is, we have a vertex partition into $s$ independent sets. Using Lemma \ref{lm1}, an independent set in $H$ (and thus in $H^*$) has size at most $\frac{\lambda n}{d}$. Otherwise, by Lemma \ref{lm1}, every vertex set of size greater than $\frac{\lambda n}{d}$ determines every color, which means there exists two vertices connected by a $d$-color edge, contradicting the independence. As a result $|V(H^{*})|\leq s\cdot \frac{\lambda n}{d}$, proving the lemma. 
\end{proof}

The next lemma develops this further by giving lower bounds on the number of disjoint copies of star graphs.

\begin{lemma}\label{main5}
Let $G$ be an $(n, d, \lambda)$-colored graph with color set $D$, and let $U\subset V(G)$ with $|U|=r\cdot \frac{\lambda n}{d}$. Then the number of vertex disjoint copies of the nonempty star graph $K_{1,m}$ with any fixed edge-coloring from $D$ is at least $\frac{r-m}{m+1} \cdot \frac{\lambda n}{d}$.  
\end{lemma}

\begin{proof}
Let $T$ be the maximal set of copies of $K_{1,m}$ in $U$, and $H$ be the union of all copies in $T$. Then $U-H$ will have no copies of $K_{1,m}$. 

Suppose the set of color of $K_{1,m}$ is $\{c_{1}, c_{2}, \dots, c_{t}\}$ with multiplicities $\{m_{1}, m_{2}, \dots, m_{t} \}$. Using Lemma \ref{main4}, for each $i$ there are at most $m_{i}\cdot \frac{\lambda n}{d}$ vertices that are incident with fewer than $m_{i}$ edges colored by $c_{i}$. Summing over $i$ we get that there are at most
\[\sum_{i=1}^{t}m_i \cdot \frac{\lambda n}{d} = m \cdot \frac{\lambda n}{d}\]
vertices of $U-H$ which are not colored $c_i$ from at least $m_i$ other vertices of $U-H$ for every $i$. If vertex $v\in U-H$ is incident with at least $m_{i}$ edges color $c_{i}$ for every $i,$ then $v$ is the singleton bipartition set of an instance of $K_{1,m}$. Thus $|U-H|\le m\cdot \frac{\lambda n}{d}$. By disjointness 
\[|T|=\frac{|H|}{m+1} \ge \frac{r\cdot \frac{\lambda n}{d}-m\cdot \frac{\lambda n}{d}}{m+1} = \frac{r-m}{m+1}\cdot \frac{\lambda n}{d} \]
as required.

\end{proof}
Our final technical lemma is a simple application of Lemma \ref{main4} that gives a lower bound on the number of disjoint edges of a given color in a vertex set.
\begin{lemma}\label{lm2}
Let $G$ be an $(n,d,\lambda)$-colored graph with color set $D$, and let $U \subset V(G)$ with $|U| \geq \frac{2\lambda n}{d}$. Then for each color $c \in D$, the number of disjoint $c$ colored edges in $U$ is at least $\frac{|U|}{2}-\frac{\lambda n}{d}$. 
\end{lemma}

\begin{proof}
We partition the vertex set of $U$ into two sets, $A$ and $B$, such that $|A|=|B|=\frac{|U|}{2}$. Choose as large a matching of color $c$ as possible between, say, $A'\subseteq A$ and $B'\subseteq B$. We have that the two sets $A\setminus A'$ and $B\setminus B'$ both have size at most $ \frac{\lambda n}{d}$. Otherwise, by Lemma \ref{lm1} we could increase the size of our matching. As a result, the number of disjoint $c$ colored edges in $U$ is at least
\[|A^{'}|=|B^{'}|\geq \frac{|U|}{2}-\frac{\lambda n}{d}\]
as required.
\end{proof}
\subsection{Proof of Theorem \ref{tree-thm1}}
The proof proceeds by strong induction on the number of edges in $T$. If $T$ contains no edges the theorem is clearly true; if $T$ is a star graph $K_{1, m}$, then $\sigma(G) = m + 1$ and the theorem is Lemma \ref{main5}. 

Now assume $T$ is not a star graph. Let $T^{'}$ be the graph produced by deleting all leaves of $T$. Since $T$ is not a star graph, $T{'}$ is a tree which has at least two leaves, we can choose $v$ be a leaf of $T^{'}$ such that there exists another leaf of $T^{'}$, say $w$, such that $\deg_{T} v \leq \deg_{T} w$. 
Suppose the set of leaves of $T$ connected to $v$ is $\{v_{1}, v_{2}, \dots, v_{y}\}$. Define the graph $T^*$ to be $T \setminus \{v_1,v_2,\dots,v_y\}$. By construction, $T^*$ is a tree with fewer edges than $T$ and $\sigma(T) = \sigma(T^*)\cdot(y + 1)$. By the inductive hypothesis
we have the number of disjoint copies of $T^*$ in $U$ denoted by $C_{T^*}$  is at least $\left(  \frac{r}{\sigma(T^*)}-1 \right)\cdot \frac{\lambda n}{d}$.

We are building our tree $T$ out of stars instead of edges. Let $W$ be the set of copies of $v$ in $U$. By disjointness $|W|=|C_{T^*}|$. Let $K_{1,y}$ be the star graph generated by $\{v, v_{1}, v_{2}, \dots, v_{y} \}$ where the root is $v$. Using Lemma \ref{main5}, there exists at least $\frac{|W|/\frac{\lambda n}{d}-y}{y+1} \cdot \frac{\lambda n}{d}$ disjoint copies of $K_{1,y}$ in $W$. For each copy of $K_{1,y}$ we can build our tree $T$ by adding the copies of $T^*$ that correspond to $v$. These are disjoint copies of $T$ because of the disjointness of $T^*$ and the disjointness of $K_{1,y}$. So there are at least       
\begin{align*}
    \frac{|W|/\frac{\lambda n}{d}-y}{y+1}\cdot \frac{\lambda n}{d}
    &=\frac{|C_{T^*}|/\frac{\lambda n}{d}-y}{y+1}\cdot \frac{\lambda n}{d}\\
    &\ge \frac{\left(\frac{r}{\sigma(T^*)}-1\right )-y}{y+1}\cdot \frac{\lambda n}{d}\\
    &= \left(\frac{r}{(y+1)\sigma(T^*)}-1 \right)\cdot \frac{\lambda n}{d}\\
    &= \left(\frac{r}{ \sigma(T)}-1 \right)\cdot \frac{\lambda n}{d}
\end{align*}
disjoint copies of $T$ as required.

\subsection{Proof of Theorem \ref{tree-thm12}}
The proof proceeds by induction on the number of edges on $T$. If $T$ contains no edges, the theorem is clearly true. If $T$ is an edge, then $|V(T)|=2$, the theorem is Lemma \ref{lm2}. 

So assume $T$ is a tree with $m$ vertices. Consider the subgraph $T^*$ of $T$ produced by deleting one leaf on vertex $x$. Let's say the edge we are just removing has color $c$. By construction $T^*$ is a tree with fewer edges than $T$, say $m-1$. By inductive hypothesis we have the collection of disjoint copies of $T^*$ in $U$ is at least $\frac{|U|}{m-1}-\frac{\lambda n}{d}$.

Choose $\frac{|U|}{m}$ copies of them arbitrarily and let this set of vertices be called $S$. This is possible since $|U|\ge m(m-1)\frac{\lambda n}{d}$. 

So $S$ has size $(m-1) \cdot\frac{|U|}{m}$. Now in these copies of $T^*$, denote by $A$ the set copies of $x$ to which we will be trying to add an edge of color $c$, so $|A|=\frac{|U|}{m}$. Let $B=U\setminus S$ so $|B|= |U|-(m-1) \cdot\frac{|U|}{m}=\frac{|U|}{m}$.

Choose as large of a matching color $c$ as possible between, say, $A'\subseteq A$ and $B'\subseteq B$, each matching creates a copy of $T$. Let $C$ and $D$ be the sets of vertices in $A\setminus A'$ and $B\setminus B'$ respectively. Then we have that $|C|=|D|\leq \frac{\lambda n}{d}$. Otherwise using Lemma \ref{lm1} we can find at least one $c$ colored edge between $C$ and $D$, which would increase the size of our matching. So the number of disjoint copies of $T$ is 
\[|A'|=|A|-|C|\geq \frac{|U|}{m}-\frac{\lambda n}{d},\]
as required.

\section{Acknowledgements}
T. Pham would like to
thank to the VIASM for the hospitality and for the excellent working conditions. M. Tait was partially supported by National Science Foundation grant DMS-2011553 and a Villanova University Summer Grant.

\bibliographystyle{amsplain}

\bibliography{bib}

\providecommand{\bysame}{\leavevmode\hbox to3em{\hrulefill}\thinspace}
\providecommand{\MR}{\relax\ifhmode\unskip\space\fi MR }
\providecommand{\MRhref}[2]{%
  \href{http://www.ams.org/mathscinet-getitem?mr=#1}{#2}
}
\providecommand{\href}[2]{#2}
\begin{thebibliography}{10}

\bibitem{BST}
Eiichi Bannai, Osamu Shimabukuro, and Hajime Tanaka, \emph{Finite {E}uclidean
  graphs and {R}amanujan graphs}, Discrete mathematics \textbf{309} (2009),
  no.~20, 6126--6134.

\bibitem{bene}
Michael Bennett, Jeremy Chapman, David Covert, Derrick Hart, Alex Iosevich, and
  Jonathan Pakianathan, \emph{Long paths in the distance graph over large
  subsets of vector spaces over finite fields}, J. Korean Math \textbf{53}
  (2016), 115--126.

\bibitem{BIT}
Michael Bennett, Alexander Iosevich, and Krystal Taylor, \emph{Finite chains
  inside thin subsets of $\mathbb{R}^d$}, Analysis and PDE \textbf{9} (2016),
  no.~3, 597--614.

\bibitem{SP1}
Jean Bourgain, Nets Katz, and Terence Tao, \emph{A sum-product estimate in
  finite fields, and applications}, Geometric and Functional Analysis
  \textbf{14} (2004), no.~1, 27--57.

\bibitem{steven}
David Covert and Steven Senger, \emph{Pairs of dot products in finite fields
  and rings}, Combinatorial and Additive Number Theory II, Springer, 2015,
  pp.~129--138.

\bibitem{EF1}
Xiumin Du, Alex Iosevich, Yumeng Ou, Hong Wang, and Ruixiang Zhang, \emph{An
  improved result for {F}alconer's distance set problem in even dimensions},
  Mathematische Annalen \textbf{380} (2021), no.~3, 1215--1231.

\bibitem{EF2}
Xiumin Du and Ruixiang Zhang, \emph{Sharp ${L}^2$ estimates of the
  {S}chr\"{o}dinger maximal function in higher dimensions}, Annals of
  Mathematics \textbf{189} (2019), no.~3, 837--861.

\bibitem{KP1}
Zeev Dvir, \emph{On the size of {K}akeya sets in finite fields}, Journal of the
  American Mathematical Society \textbf{22} (2009), no.~4, 1093--1097.

\bibitem{EIT}
Suresh Eswarathasan, Alex Iosevich, and Krystal Taylor, \emph{Fourier integral
  operators, fractal sets, and the regular value theorem}, Advances in
  Mathematics \textbf{228} (2011), no.~4, 2385--2402.

\bibitem{SP2}
Moubariz~Z. Garaev, \emph{An explicit sum-product estimate in $\mathbb{F}_p$},
  International Mathematics Research Notices \textbf{2007} (2007).

\bibitem{Hart}
Derrick Hart, Alex Iosevich, Doowon Koh, and Misha Rudnev, \emph{Averages over
  hyperplanes, sum-product theory in vector spaces over finite fields and the
  {E}rd{\H{o}}s-{F}alconer distance conjecture}, Transactions of the American
  Mathematical Society \textbf{363} (2011), no.~6, 3255--3275.

\bibitem{HIS}
Derrick Hart, Alex Iosevich, and Jozsef Solymosi, \emph{Sum-product estimates
  in finite fields via {K}loosterman sums}, International Mathematics Research
  Notices \textbf{2007} (2007), 9.

\bibitem{IJM}
Alex Iosevich, Gail Jardine, and Brian McDonald, \emph{Cycles of arbitrary
  length in distance graphs on $\mathbb{{F}}_q^d$}, Proceedings of the Steklov
  Institute of Mathematics 314 (1), 2021, pp.~27--43.

\bibitem{RT5}
Alex Iosevich, Doowon Koh, and Mark Lewko, \emph{Finite field restriction
  estimates for the paraboloid in high even dimensions}, accepted in Journal of
  Functional Analysis \textbf{2019}.

\bibitem{FT1}
Alex Iosevich, Mihalis Kolountzakis, Yurii Lyubarskii, Azita Mayeli, and
  Jonathan Pakianathan, \emph{On {G}abor orthonormal bases over finite prime
  fields}, Bulletin of the London Mathematical Society \textbf{53} (2021),
  no.~2, 380--391.

\bibitem{FT2}
Alex Iosevich, Azita Mayeli, and Jonathan Pakianathan, \emph{The {F}uglede
  conjecture holds in $\mathbb {Z} _p\times\mathbb {Z} _p$}, Anal. PDE
  \textbf{10} (2017), no.~4, 757--764.

\bibitem{EF3}
Alex Iosevich and Misha Rudnev, \emph{Erd{\H{o}}s distance problem in vector
  spaces over finite fields}, Transactions of the American Mathematical Society
  \textbf{359} (2007), no.~12, 6127--6142.

\bibitem{KLS2}
Henryk Iwaniec and Emmanuel Kowalski, \emph{Analytc number theory colloquium
  publications}, vol.~53, American Mathematical Soc., 2004.

\bibitem{koh2021configurations}
Doowon Koh, Sujin Lee, Thang Pham, and Chun-Yen Shen, \emph{Configurations of
  rectangles in a set in $\mathbb {F} _q^ 2$}, arXiv preprint arXiv:2103.11419
  (2021).

\bibitem{MSS}
Michael Krivelevich and Benny Sudakov, \emph{Pseudo-random graphs}, More sets,
  graphs and numbers (Heidelberg, ed.), Springer, Berlin, 2006, pp.~199--262.

\bibitem{kwok}
Wing~Man Kwok, \emph{Character tables of association schemes of affine type},
  European journal of combinatorics \textbf{13} (1992), no.~3, 167--185.

\bibitem{RT3}
Mark Lewko, \emph{New restriction estimates for the 3-d paraboloid over finite
  fields}, Adv. Math. \textbf{270} (2015), no.~no. 1, 457--479.

\bibitem{RT2}
Mark Lewko, \emph{Finite field restriction estimates based on {K}akeya maximal
  operator estimates}, Journal of the European Mathematical Society \textbf{21}
  (2019), no.~12, 3649--3707.

\bibitem{liu2020hausdorff}
Bochen Liu, \emph{On {H}ausdorff dimension of radial projections}, Revista
  Matem{\'a}tica Iberoamericana \textbf{37} (2020), no.~4, 1307--1319.

\bibitem{LM1}
Neil Lyall and {\'A}kos Magyar, \emph{Weak hypergraph regularity and
  applications to geometric {R}amsey theory}, Transactions of the American
  Mathematical Society, Series B \textbf{9} (2022), no.~5, 160--207.

\bibitem{merris1997multilinear}
Russell Merris, \emph{Multilinear algebra}, {C}{R}{C} Press, 1997.

\bibitem{RT1}
Gerd Mockenhaupt and Terence Tao, \emph{Restriction and {K}akeya phenomena for
  finite fields}, Duke Math J. \textbf{121} (2004), no.~1, 35--74.

\bibitem{RT4}
Misha Rudnev and Ilya~D Shkredov, \emph{On the restriction problem for discrete
  paraboloid in lower dimension}, Advances in Mathematics \textbf{339} (2018),
  657--671.

\bibitem{david}
David~M. Soukup, \emph{Embeddings of weighted graphs in {E}rd{\H{o}}s-type
  settings}, Moscow Journal of Combinatorics and Number Theory \textbf{8}
  (2019), no.~2, 117--123.

\bibitem{Vinh}
Le~Anh Vinh, \emph{The solvability of norm, bilinear and quadratic equations
  over finite fields via spectra of graphs}, Forum mathematicum \textbf{1}
  (2014), no.~26.

\bibitem{KLS1}
Andr{\'e} Weil, \emph{On some exponential sums}, Proc. Nat. \textbf{34} (1948),
  204--207, Acad. Sci. U.S.A.

\bibitem{KP2}
Thomas Wolff, \emph{Recent work connected with the {K}akeya problem}, Prospects
  in mathematics (Princeton, NJ, 1996) \textbf{2} (1999), 129--162.

\end{thebibliography}

\end{document}